\documentclass[11pt]{amsart}
\usepackage{times,amsmath}
\usepackage{amsthm}
\usepackage{amsmath,algorithm2e,bm}
\usepackage{amsmath,amssymb,amsfonts,amscd}
\usepackage{graphicx,latexsym,color}
\usepackage{hyperref}
\usepackage{multirow}
\usepackage{floatrow}
\usepackage{wrapfig}

\usepackage{verbatim}
\usepackage{cite}
\usepackage{subcaption}
\usepackage[pagewise]{lineno}

\allowdisplaybreaks[4]

\newcommand{\p}{\partial}

\newcommand{\dif}{\,\mathrm{d}}

\newtheorem{thm}{Theorem}[section]
\newtheorem{lem}[thm]{Lemma}

\newtheorem{remark}[thm]{Remark}
\def \bes{\begin{eqnarray}}
\def \ees{\end{eqnarray}}
\def \bns{\begin{eqnarray*}}
\def \ens{\end{eqnarray*}}

\newenvironment{eqa}{\begin{equation}%
  \begin{array}{rcl}}{\end{array}\end{equation}}
\newcommand\beqa{\begin{eqa}}
\newcommand\eeqa{\end{eqa}}
\usepackage[bottom=1.2in,top=1in,left=1in,right=1in]{geometry}

\begin{document}
\numberwithin{equation}{section}

\title{A free-boundary model of vascularized tumor growth with time-periodic coefficients}

\author{Xinyue Evelyn Zhao}  
\address{Department of Mathematics, University of Tennessee, Knoxville, TN 37916, USA, corresponding author} 
\email{xzhao45@utk.edu}

\begin{abstract}
We investigate a free boundary PDE model for vascularized tumor growth with two time-periodic coefficients representing nutrient supply $\phi(t)$ and nutrient demand $\psi(t)$. Under radial symmetry, the model reduces to a nonautonomous ODE for the tumor radius. We prove a vanishing--persistence dichotomy governed by the averaged coefficients $\bar{\phi}$ and $\bar{\psi}$: if the average nutrient supply does not exceed the average nutrient demand, namely $\bar{\phi}\le \bar{\psi}$, the tumor radius tends to zero; if the average supply exceeds the average demand, namely $\bar{\phi}>\bar{\psi}$, the tumor persists and converges to a unique positive periodic solution. We also study the linear stability of this periodic solution when nonradial perturbations are imposed. Under a stronger condition on the nutrient supply and demand, we show that the radially symmetric periodic solution is linearly stable when the tumor aggressiveness parameter is sufficiently small. Numerical simulations are provided to illustrate the analytical results.
\end{abstract}

\maketitle

\section{Introduction}
Tumor growth is a complex biological process driven by the interplay between cell proliferation, nutrient supply, and interactions with the surrounding tissue, often resulting in irregular and evolving geometries. From a mathematical perspective, such evolving tumor shapes can be described by PDE models in the form of free boundary problems; see, for example, \cite{bazaliy2003global,bazaliy2003free,cui2009well,cui2007bifurcation,fontelos2003symmetry,friedman2007mathematical,friedman2006bifurcation,friedman2006asymptotic,friedman2008stability,friedman1999analysis,friedman2001symmetry, FFH1,angio2,angio1,huang2021asymptotic,he2022linear,he2022thelinear,zhao2020impact,zhao2020symmetry, cui2001analysis,hao2012bifurcation,zhao2025analysis,zhao2025determination,wu2025optimal,zhao2024optimal}, and the references therein. These models typically regard the tumor as a continuum of proliferating cells with constant density and describe the dynamics of nutrient concentration, internal pressure, and boundary motion through reaction-diffusion equations, conservation laws, and geometric evolution equations.

In this paper, we consider a vascularized tumor. Angiogenesis, the process by which new blood vessels form from existing ones, plays a crucial role in tumor development. Prior to angiogenesis, tumor cells receive nutrients primarily through diffusion from the host tissue, and the tumor size is therefore limited by the diffusion range of nutrients. Consequently, nonvascularized tumors are typically slow-growing and remain on the order of a few millimeters in size \cite{lowengrub2,spill2015mesoscopic}. As the tumor grows, the existing vasculature becomes insufficient to supply nutrients to all cells, which may trigger angiogenesis. Tumor cells then secrete signaling factors that stimulate the surrounding vasculature to grow toward the tumor. Once vascularized, the tumor develops its own capillary network, allowing nutrients to be supplied more efficiently and enabling further growth. This transition is often associated with increased malignancy and the potential for metastasis. Mathematical analysis of such models can provide insight into tumor growth mechanisms and may contribute to the development of therapeutic strategies.

Let $\hat{u}(\hat{\bm x},t)$ denote nutrient concentration. Nutrients (such as glucose) diffuse through the tumor tissue and are taken up by tumor cells. Since the rate of nutrient diffusion (e.g., for glucose, on the order of minutes) is much faster than the rate of tumor growth (on the order of days or months), the nutrient concentration can be approximated as being in quasi-steady state for a given tumor morphology. Therefore, $\hat{u}$ satisfies:
\begin{equation}\label{eqn:uhat}
0= \Delta \hat{u} + A(u_B(t)-\hat{u}) - \lambda_0 \hat{u}  \qquad \text{in }\Omega(t),
\end{equation}
where $\Omega(t)\subseteq \mathbb{R}^2$ denotes the evolving tumor region, $A(u_B(t) -\hat{u})$ represents the supply of nutrients from the vasculature, $u_B(t)$ is the nutrient concentration in the blood vessels, and $\lambda_0$ is the consumption rate by tumor cells.

In vascularized tumor tissue, nutrients are supplied through the tumor vasculature, so the flux across the tumor boundary depends on the difference between the nutrient concentration inside and outside the tumor. Accordingly, $\hat{u}$ satisfies the Robin boundary condition (see \cite{angio1,angio2}):
\begin{equation}\label{eqn:uhatbdy}
\frac{\p \hat{u}}{\p \bm{n}} + \hat{\alpha} \bigl(\hat{u} - \bar{u}(t) \bigr) = 0 \qquad \text{on }\p \Omega(t),
\end{equation}
where $\bm{n}$ is the outward normal, $\bar{u}(t)$ denotes the nutrient concentration outside the tumor, and $\hat{\alpha}>0$ is a parameter reflecting the extent of angiogenesis. In general, $\hat{\alpha}$ may depend on time; however, for analytical simplicity, we assume it is a constant.

As tumor cells proliferate and die, the assumption of constant cell density implies that cell movement generates internal pressure $\hat{p}(\hat{\bm x},t)$. Following \cite{ friedman1999analysis, cristini2003nonlinear}, we assume that the tumor tissue behaves as a porous medium. Let $\bm{V}$ denote the velocity field of tumor cell motion. Then, by Darcy's law
\[
\bm{V} = -\nabla \hat{p}.
\]
Mass conservation then yields
\begin{equation}\label{eqn:p}
    - \Delta \hat{p}  = \text{div} \bm{V} = \hat{\mu}(\hat{u} - \tilde{u}) \qquad \text{in }\Omega(t),
\end{equation}
where $\tilde{u}$ is the critical nutrient concentration below which cells cannot survive, and the net proliferation rate of the tumor is assumed to be proportional to $\hat{u}-\tilde{u}$. The parameter $\hat{\mu} >0$ represents the aggressiveness of the tumor: it measures the rate of tumor expansion due to cell proliferation when $\hat{u}>\tilde{u}$, and tumor shrinkage due to cell death when $\hat{u}<\tilde{u}$.

We then specify the boundary condition for the pressure. As in \cite{bazaliy2003global,friedman2007mathematical,friedman2006asymptotic,friedman2008stability}, the cell-to-cell adhesiveness gives rise to the boundary condition:
\begin{equation}\label{eqn:pbdy}
    \hat{p} = \kappa \qquad \text{on }\partial \Omega(t),
\end{equation}
where $\kappa$ denotes the mean curvature of the boundary. Furthermore, assuming the velocity field given by Darcy's law is continuous up to the boundary, the normal velocity of the free boundary satisfies:
\begin{equation}\label{eqn:bdy}
    \hat{V}_n =\bm{V}\cdot \bm{n} =  -\frac{\partial \hat{p}}{\partial \bm{n}} \qquad \text{on }\partial \Omega(t)
\end{equation}
where $\hat{V}_n$ denotes the normal velocity of the moving boundary.

To simplify the system, we introduce the change of variables (see \cite{angio2}):
\begin{eqnarray}
    &\displaystyle \label{change0} \lambda = A+ \lambda_0, \qquad \bm{x} = \sqrt{\lambda}\hat{\bm x},\qquad  \alpha = \frac{\hat{\alpha}}{\sqrt{\lambda}}, \qquad \mu = \frac{\hat{\mu}}{\lambda}, \qquad V_n = \frac{\hat{V}_n}{\sqrt{\lambda}},\\
    &\displaystyle \label{change1} u(\bm{x},t) = \hat{u}(\hat{\bm x},t) -\frac{u_B(t)}{1+\lambda_0/A},\qquad p(\bm{x},t) = \hat{p}(\hat{\bm x},t),\\
    &\displaystyle \label{change2} \phi(t) := {\bar u(t)} - \frac{u_B(t)}{1+\lambda_0/A},\\
    &\displaystyle \label{change3}  \psi(t) := \tilde{u}-\frac{u_B(t)}{1+\lambda_0/A}.
\end{eqnarray}
Under this transformation, the nutrient equation becomes
\begin{eqnarray}
    &\displaystyle \label{seqn1} 0 = \Delta u -   u \qquad \text{in }\Omega(t),\\
    &\displaystyle \label{seqn2} \frac{\p u}{\p \bm{n}} + \alpha (u - \phi(t)) = 0 \qquad \text{on }\p \Omega(t),
\end{eqnarray}
while the pressure equation reduces to
\begin{eqnarray}
    &\displaystyle \label{seqn3} -\Delta p = \mu(u - \psi(t)) \qquad \text{in }\Omega(t),\\
    &\displaystyle \label{seqn4} p = \kappa \qquad \text{on } \partial \Omega(t),
\end{eqnarray}
and the free boundary equation becomes
\begin{equation}
    \label{seqn5} V_n = -\frac{\p p}{\p {\bm n}} \qquad \text{on } \partial \Omega(t).
\end{equation}
At time $t=0$, we prescribe the initial tumor domain $\Omega(0)$ and the initial nutrient distribution:
\begin{equation}
    \label{seqn6}
    u(\bm{x},0) = \sigma_0(\bm{x})  \qquad \text{in } \Omega(0).
\end{equation}

It is noted that since $\tilde{u}$ represents the critical nutrient concentration, evaluating \eqref{eqn:uhat} at $\tilde{u}$, it is reasonable to assume
\[
A(u_B(t) - \tilde{u}) - \lambda_0 \tilde{u} < 0 \qquad \text{for all }t>0,
\]
which is equivalent to
\[
\tilde{u} > \frac{u_B(t)}{1+\lambda_0/A} \qquad \text{for all }t>0.
\]
In addition, we assume $\bar{u}(t) > \tilde{u}$, which implies
\[
\bar{u}(t) > \frac{u_B(t)}{1+\lambda_0/A} \qquad \text{for all }t>0.
\]
Consequently, the functions $\phi(t)$ and $\psi(t)$ defined in \eqref{change2}–\eqref{change3} are strictly positive.

Motivated by periodic food intake and the resulting cyclic variation in nutrient
concentration in both the vasculature and the surrounding normal tissue, we
assume that $u_B(t)$ and $\bar u(t)$ are positive $T$-periodic functions, where
$T$ may be interpreted biologically as a daily period. By \eqref{change2} and
\eqref{change3}, the transformed coefficients $\phi(t)$ and $\psi(t)$ are also
$T$-periodic. For convenience, we introduce the following notations
\begin{eqnarray}
    && \displaystyle \label{phi} \bar{\phi} := \frac1T \int_0^T \phi(t) \dif t, \qquad \phi^* := \max\limits_{t\in [0,T]} \phi(t), \qquad \phi_* := \min\limits_{t\in [0,T]} \phi(t),\\
    && \displaystyle \label{psi} \bar{\psi} := \frac1T \int_0^T \psi(t) \dif t, \qquad \psi^* := \max\limits_{t\in [0,T]} \psi(t), \qquad \psi_* := \min\limits_{t\in [0,T]} \psi(t).
\end{eqnarray}
These quantities play a crucial role in the analysis of the long-time dynamics.

Time-periodic nutrient supply has already been considered in free boundary tumor models. Bai and Xu \cite{bai2013qualitative}, Huang, Zhang, and Hu
\cite{huang2021asymptotic}, and He and Xing \cite{he2021existence} studied
nonvascularized free boundary tumor models with periodic external nutrient
supply. These models may be viewed as special cases of the present framework
when the threshold $\psi(t)\equiv \tilde{u}$ is constant and, formally, the Robin boundary
condition reduces to the Dirichlet boundary condition in the limit
$\alpha\to\infty$. In the radially symmetric case, Bai and Xu
\cite{bai2013qualitative} showed that if $\bar\phi<\tilde u$, then the
tumor-free state is globally stable, while if $\phi_*>\tilde u$, then there
exists a unique positive $T$-periodic solution that is stable under radial
perturbations. Huang, Zhang, and Hu \cite{huang2021asymptotic} further studied
the linear stability of this periodic solution under nonradial perturbations. He and Xing \cite{he2021existence} addressed the remaining regime
$\tilde u\in[\phi_*,\bar\phi)$ and proved that a positive $T$-periodic solution
exists whenever $\bar\phi>\tilde u$. Moreover, Huang and Hu
\cite{huang2024periodic} studied the corresponding model without the
quasi-steady-state approximation. More recently, Liu and Hu
\cite{liu2025periodic} extended the periodic nutrient framework to a
vascularized tumor model in which the angiogenesis coefficient is related to
permeability and is inversely proportional to the tumor radius.

In contrast, the present work considers a vascularized tumor model with two
time-periodic coefficients. We allow not only the nutrient input but also the
effective threshold for net proliferation to vary periodically in time. In the radially symmetric setting, we derive a vanishing--persistence
dichotomy governed by the averages $\bar\phi$ and $\bar\psi$.

\begin{thm}\label{dich}
Consider the system \eqref{seqn1} -- \eqref{seqn6} in the radially symmetric
case, where the tumor domain is a disk $\Omega(t)=B_{R(t)}$ centered at the origin.
\begin{itemize}
    \item If $\bar{\phi}\le \bar{\psi}$, then the tumor vanishes, namely
    \[
    \lim_{t\to\infty}R(t)=0.
    \]
    \item If $\bar{\phi}>\bar{\psi}$, then the tumor persists and approaches
    a positive $T$-periodic state. More precisely, there exists a unique
    positive $T$-periodic solution $R_*(t)$, and every solution with
    $R(0)=R_0>0$ satisfies
    \[
    \lim_{t\to\infty}|R(t)-R_*(t)|=0.
    \]
\end{itemize}
\end{thm}

Biologically, the results of Theorem \ref{dich} show that the long-time fate of the tumor is governed by the balance between the average nutrient supply and the average nutrient demand. In the model \eqref{seqn1} -- \eqref{seqn6}, the function $\phi(t)$ represents the nutrient supply available to the tumor, while $\psi(t)$ represents the nutrient threshold, or demand, required for net tumor growth. If the average supply does not exceed the average demand, namely $\bar{\phi}\le \bar{\psi}$, then the tumor cannot sustain long-term growth, and its radius tends to zero. On the other hand, if the average supply exceeds the average demand, namely $\bar{\phi}>\bar{\psi}$, then the tumor persists and eventually approaches a stable periodic pattern which synchronizes the nutrient cycles.

Moreover, we study the linear stability of the radially symmetric periodic solution under nonradial perturbations of the form 
\[
\partial \Omega(t): r=R_*(t) + \varepsilon \rho(\theta,t).
\]
Under a stronger assumption $\phi_*>\psi^*$, we prove that the periodic radial solution is linearly stable for sufficiently small tumor
aggressiveness parameter $\mu$.

\begin{thm}\label{linearstable}
    Assume that $\phi_* > \psi^*$. Then there exists $\mu_*>0$ such that, for $0<\mu<\mu_*$, the radially symmetric $T$-periodic solution $(u_*(r,t),\,p_*(r,t),\,R_*(t))$ is linearly stable in the sense that there exist constants $a_1,b_1,\delta,C$, such that
    \begin{equation}
        \label{linearstability}
        |\rho(\theta,t) - \big(a_1 \cos(\theta) + b_1 \sin(\theta)\big)| \le C e^{-\delta t},
    \end{equation}
    as $t\rightarrow \infty$.
\end{thm}

\begin{remark}
    The $n=1$ mode is excluded in \eqref{linearstability} because it corresponds to a translation of the tumor center. 
\end{remark}

\begin{remark}
The stronger assumption $\phi_*>\psi^*$ is used only in the linear stability
analysis, where it provides $\mu$-independent positive lower and upper bounds
for $R_*(t;\mu)$. Under the weaker averaged condition $\bar\phi>\bar\psi$, the
existence and uniqueness of the positive $T$-periodic solution still hold.
However, this averaged condition alone does not yield a $\mu$-independent
positive lower bound for $R_*(t;\mu)$. The linear stability of $(u_*(r,t),\,p_*(r,t),\,R_*(t))$ under only the
averaged condition $\bar\phi>\bar\psi$ remains open.
\end{remark}

The structure of this paper is as follows. In Section 2, we study the radially symmetric case and prove Theorem \ref{dich}, treating the vanishing and persistence cases separately. In Section 3, we linearize \eqref{seqn1} -- \eqref{seqn6} around the radially symmetric $T$-periodic solution and prove Theorem \ref{linearstable}. In Section 4, we present numerical simulations to illustrate our theoretical results.

\section{Radially Symmetric Case}
We consider the radially symmetric case, in which the tumor domain is a disk of the form
\[
\Omega(t)=\{\bm{x}\in\mathbb R^2:\ |\bm{x}|<R(t)\},
\]
where $R(t)$ denotes the tumor radius, and the nutrient concentration and pressure are radially symmetric, namely
\[
u(\bm{x},t)=u(r,t),\qquad p(\bm{x},t)=p(r,t),\qquad r=|\bm{x}|.
\]
Under the radial symmetry assumption, the free boundary problem \eqref{seqn1} -- \eqref{seqn5} reduces to a 1-D moving boundary problem for $(u(r,t),p(r,t),R(t))$. The governing system then becomes
\begin{eqnarray}
& \displaystyle \label{eqn2.1} 0 = u_{rr}+\frac1r u_r- u  \qquad 0<r<R(t),\ t>0,\\
& \displaystyle \label{eqn2.2} u_r+\alpha\bigl(u-\phi(t)\bigr) = 0  \qquad r=R(t),\ t>0,\\
& \displaystyle \label{eqn2.3} -\left(p_{rr}+\frac1r p_r\right) = \mu\bigl(u-\psi(t)\bigr) \qquad 0<r<R(t),\ t>0,\\
& \displaystyle \label{eqn2.4} p  =  \frac1{R(t)}  \qquad r=R(t),\ t>0,\\
& \displaystyle \label{eqn2.5} R'(t) = -p_r(R(t),t) \qquad t>0.
\end{eqnarray}
For this reduced system, $u(r,t)$ can be solved from \eqref{eqn2.1} -- \eqref{eqn2.2}:
\begin{equation}\label{sol:u}
    u(r,t) = \frac{\alpha \phi(t)}{I_1(R(t))+\alpha I_0(R(t))}I_0(r),
\end{equation}
where $I_0(r)$ and $I_1(r)$ are the modified Bessel functions of the first kind. Multiplying \eqref{eqn2.3} by $r$, integrating over $(0,R(t))$, and using \eqref{eqn2.5}, we obtain
\[
\frac{\dif R}{\dif t} = \frac{\mu}{R(t)}\int_0^{R(t)} (u-\psi(t)) r \dif r.
\]
Substituting \eqref{sol:u} into it yields
\begin{equation}\label{eqn:R}
\begin{split}
    \frac{\dif R}{\dif t} &= \frac{\mu}{R(t)} \left(\frac{\alpha \phi(t)}{I_1(R(t))+\alpha I_0(R(t))} \int_0^{R(t)}  I_0(r) r \dif r -\psi(t) \int_0^{R(t)}  r \dif r \right) \\
    &= \frac{\mu}{R(t)}  \left( \frac{\alpha \phi(t)}{I_1(R(t)) + \alpha I_0(R(t))}R(t) I_1(R(t)) -  \frac{\psi(t)}2 R^2(t) \right)\\
    &= \mu R(t) \left( f_0(R(t))\phi(t) - \frac12 \psi(t)\right),
    \end{split}
\end{equation}
where $f_0(r)$ is defined as
\begin{equation}\label{f0}
    f_0(r) := \frac{\alpha I_1(r)}{r(I_1(r) + \alpha I_0(r))}.
\end{equation}
Hence, the radially symmetric free-boundary problem is reduced to the following ODE with two time-periodic coefficients:
\begin{equation}\label{sys1}
    \frac{\dif R}{\dif t} = \mu R\left( f_0(R)\phi(t) - \frac12 \psi(t)\right),
\end{equation}
and we prescribe a positive initial condition
\begin{equation}\label{sys2}
    R(0) = R_0 > 0.
\end{equation} 

We shall use the following property of $f_0$; the proof can be found in Lemma 2.1 in \cite{liu2025periodic}.

\begin{lem} \label{lem:f0}
    The function $f_0(r)$ is decreasing in $r$ for $0<r<\infty$ and $f_0(0^+) = \frac12$ and $f_0(+\infty) = 0$.
\end{lem}

We begin by establishing the global well-posedness of the initial value problem \eqref{sys1} -- \eqref{sys2}.

\begin{thm}\label{thm:pos}
For each $R_0>0$, the initial value problem \eqref{sys1}--\eqref{sys2} admits a unique global solution
\[
R \in C^1([0,\infty)).
\]
Moreover, this solution satisfies
\[
R(t)>0 \qquad \text{for all } t\ge 0.
\]
\end{thm}

\begin{proof}
Define
\[
F(t,R):=\mu R\left(f_0(R)\phi(t)-\frac12\psi(t)\right), \qquad t\ge 0,\; R>0.
\]
Since $I_0$ and $I_1$ are smooth and $I_1(R)+\alpha I_0(R)>0$ for $R>0$, the map $F$ is continuous in $t$ and locally Lipschitz continuous in $R$ on $(0,\infty)$. Hence, by the Picard--Lindel\"of theorem, there exists a unique maximal solution
\[
R \in C^1([0,\tau_{\max}))
\]
of \eqref{sys1}--\eqref{sys2}, where $\tau_{\max}\in(0,\infty]$.

By Lemma \ref{lem:f0}, we have
\[
0<f_0(R)< \frac12 \qquad \text{for all } R>0.
\]
Therefore,
\begin{equation}
    \label{add1}
    -\frac12\psi^* < f_0(R(t))\phi(t)-\frac12\psi(t)
< \frac12(\phi^*-\psi_*)
\qquad \text{for all } t\in[0,\tau_{\max}).
\end{equation}

We first show that $R(t)$ remains positive on $[0,\tau_{\max})$. Suppose, to the contrary, that there exists $t_0\in(0,\tau_{\max})$ such that $R(t_0)=0$. Set
\[
\tau:=\inf\{t\in(0,\tau_{\max}) : R(t)=0\}.
\]
Then, $R(t)>0$ for all $t\in[0,\tau)$. By \eqref{add1}, we have, on $[0,\tau)$, 
\[
-\frac{\mu}{2}\psi^* <  \frac{R'(t)}{R(t)} = \mu\left(f_0(R(t))\phi(t)-\frac12\psi(t)\right)
< \frac{\mu}{2}(\phi^*-\psi_*).
\]
Integrating over $[0,t]$ for $0\le t<\tau$, we obtain
\[
-\frac{\mu}{2}\psi^* t
<  \ln \frac{R(t)}{R_0}
< \frac{\mu}{2}(\phi^*-\psi_*)t,
\]
or equivalently,
\[
R_0 e^{-\frac{\mu}{2}\psi^* t}
< R(t)
<  R_0 e^{\frac{\mu}{2}(\phi^*-\psi_*)t},
\qquad 0\le t<\tau.
\]
Letting $t\rightarrow \tau^-$, we get
\[
R(\tau) >  R_0 e^{-\frac{\mu}{2}\psi^*\tau}>0,
\]
which contradicts the definition of $\tau$. Thus,
\[
R(t)>0 \qquad \text{for all } t\in[0,\tau_{\max}).
\]

Next we prove that $\tau_{\max}=\infty$. If $\tau_{\max}<\infty$, then the above estimate yields
\[
R_0 e^{-\frac{\mu}{2}\psi^*\tau_{\max}}
< R(t) <
R_0 e^{\frac{\mu}{2}(\phi^*-\psi_*)\tau_{\max}}
\qquad \text{for all } t\in[0,\tau_{\max}),
\]
so the solution stays in a compact subset of $(0,\infty)$ on its maximal interval of existence. Since $F$ is locally Lipschitz in $R$ on $(0,\infty)$, the standard continuation theorem for ODEs implies that $R$ can be extended beyond $\tau_{\max}$, contradicting the maximality of $\tau_{\max}$. Therefore, $\tau_{\max}=\infty$.

Hence the initial value problem \eqref{sys1}--\eqref{sys2} admits a unique global solution, and this solution is strictly positive for all $t\ge 0$.
\end{proof}

We next divide the analysis into two cases according to the relation between $\bar{\phi}$ and $\bar{\psi}$: the vanishing case and the persistence case.

\subsection{Vanishing case.}
\begin{thm}\label{thm:vanish}
If $\bar{\phi}\le \bar{\psi}$, then the unique positive solution of
\eqref{sys1}--\eqref{sys2} satisfies
\begin{equation}\label{lim1}
\lim_{t\rightarrow\infty} R(t)=0.
\end{equation}
\end{thm}

\begin{proof}
By Theorem \ref{thm:pos}, the solution $R(t)$ is well defined and strictly
positive for all $t\ge0$.

\vspace{5pt}
We divide the proof into two cases:

\vspace{5pt}
\noindent
\textbf{Case 1: $\bar{\phi}<\bar{\psi}$.}
Let $\tau\in[0,T]$ and $n\in\mathbb{N}$. Since $0<f_0(r)< \frac12$ for all
$r>0$ by Lemma \ref{lem:f0}, we have
\[
\frac{R'(t)}{R(t)}
=\mu\left(f_0(R(t))\phi(t)-\frac12\psi(t)\right)
< \frac{\mu}{2}\bigl(\phi(t)-\psi(t)\bigr).
\]
Integrating over $[\tau,\tau+nT]$, we obtain
\[
\begin{aligned}
\ln \frac{R(\tau+nT)}{R(\tau)}
&=\int_\tau^{\tau+nT}\frac{R'(t)}{R(t)}\,\dif t \\
&< \frac{\mu}{2}\int_\tau^{\tau+nT}\bigl(\phi(t)-\psi(t)\bigr)\,\dif t = \frac{\mu nT}{2}\,(\bar{\phi}-\bar{\psi}),
\end{aligned}
\]
where the last equality follows from the $T$-periodicity of $\phi$ and $\psi$.
Hence,
\begin{equation*}
    R(\tau+nT) < R(\tau)e^{\frac{\mu nT}{2}(\bar{\phi}-\bar{\psi})}.
\end{equation*}
Set
\[
M:=\max_{0\le t\le T}R(t).
\]
Then, for any $t\ge0$, writing $t=nT+\tau$ with $\tau\in[0,T)$, we get
\[
0<R(t)<  M e^{\frac{\mu nT}{2}(\bar{\phi}-\bar{\psi})}.
\]
Since $\bar{\phi}-\bar{\psi}<0$, the right-hand side tends to $0$ as
$n\rightarrow\infty$. Therefore, \eqref{lim1} holds.

\vspace{10pt}
\noindent
\textbf{Case 2: $\bar{\phi}=\bar{\psi}$.} For any $t\ge 0$, following the same procedures as in Case 1, we have
\[
\begin{aligned}
\ln\frac{R(t+T)}{R(t)}
&= \int_t^T \frac{R'(s)}{R(s)} \dif s  \\
&< \frac{\mu}{2}\int_t^{t+T} \bigl(\phi(s) - \psi(s)\bigr) \dif s = \frac{\mu T}{2}\bigl(\bar{\phi}-\bar{\psi}\bigr) = 0,
\end{aligned}
\]
which implies 
\begin{equation}\label{Rprop}
    R(t+T) < R(t) \qquad \text{for all } t\ge0.
\end{equation}

We first claim that
\begin{equation}\label{infest}
    \liminf_{t\rightarrow\infty}R(t)=0.
\end{equation}
Assume, to the contrary, that
\[
\liminf_{t\rightarrow\infty}R(t)=a>0.
\]
Then, there exists $t_1>0$ such that
\begin{equation} \label{contradict}
R(t) > \frac{a}2 \qquad \text{for all } t \ge t_1.
\end{equation}
Since $f_0$ is decreasing,
\[
f_0(R(t))\le f_0\!\left(\frac{a}{2}\right)<\frac12
\qquad \text{for all } t \ge t_1.
\]
With this, we derive a more refined estimate than \eqref{Rprop}. For all $t  \ge t_1$
\[
\begin{aligned}
\ln\frac{R(t+T)}{R(t)}
= \int_t^T \frac{R'(s)}{R(s)} \dif s &= \mu \int_t^{t+T} \left( f_0(R(s)) \phi (s) - \frac12 \psi(s)\right) \dif s\\
&= \mu \int_t^{t+T} \left( f_0(R(s)) - \frac12 \right)\phi(s) \dif s \, + \, \frac{\mu T}{2}\left(\bar{\phi} - \bar{\psi}\right)\\
&\le \mu \int_t^{t+T} \left( f_0\!\left(\frac{a}{2}\right) - \frac12 \right)\phi(s) \dif s \, + \, \frac{\mu T}{2}\left(\bar{\phi} - \bar{\psi}\right)\\
&= -\mu \left( \frac12 - f_0\!\left(\frac{a}{2}\right)  \right) T \bar{\phi}.
\end{aligned}
\]
Hence,
\[
R(t+T) \le R(t) e^{-\mu \left( \frac12 - f_0\!\left(\frac{a}{2}\right) \right)T \bar{\phi}} \qquad \text{for all }t \ge t_1.
\]
Iterating this inequality, we obtain, for all $t\ge t_1$,
\[
R(t+nT)\le R(t)e^{-\mu \left( \frac12 - f_0\!\left(\frac{a}{2}\right)  \right)nT \bar{\phi}} \rightarrow 0 
\qquad \text{as } n \rightarrow \infty,
\]
which contradicts \eqref{contradict}. Therefore, \eqref{infest} holds.

Finally, we show that
\[
\lim_{t\to\infty}R(t)=0.
\]
Let $\varepsilon>0$ be arbitrary. From \eqref{infest}, there exists a sequence $\{t_n\}_{n=1}^\infty$ such that
\[
t_n\to\infty
\qquad\text{and}\qquad
R(t_n)\to0
\quad\text{as } n\to\infty.
\]
Hence, we may choose $N$ sufficiently large so that
\[
R(t_N)<\varepsilon e^{-\frac{\mu}{2}\phi^*T}.
\]
Now let $t\ge t_N$ be arbitrary. Then, there exists an integer $n_0\ge0$ such that 
\[
t \in [t_N+n_0 T, t_N+(n_0+1)T).
\]
Since \eqref{Rprop} remains valid for any $t\ge 0$, we have
\[
R(t_N+n_0T) <  R(t_N) < \varepsilon e^{-\frac{\mu}{2}\phi^*T}.
\]
On the other hand, for $s\in[t_N+n_0T,t]$,
\begin{equation}\label{inter1}
    \frac{R'(s)}{R(s)}
=\mu\left(f_0(R(s))\phi(s)-\frac12\psi(s)\right)
\le \frac{\mu}{2}\phi^*.
\end{equation}
Integrating over $[t_N+n_0T,t]$, we obtain
\begin{equation}\label{inter2}
\ln\frac{R(t)}{R(t_N+n_0T)}
\le \frac{\mu}{2}\phi^*(t-t_N-n_0T)
\le \frac{\mu}{2}\phi^*T,
\end{equation}
and thus
\begin{equation}\label{inter3}
R(t)\le R(t_N+n_0T)e^{\frac{\mu}{2}\phi^*T}
< R(t_N) e^{\frac{\mu}{2}\phi^*T}
<\varepsilon.
\end{equation}

Since $\varepsilon>0$ was arbitrary, it follows that
\[
\lim_{t\to\infty}R(t)=0.
\]
This completes the proof.
\end{proof}

\subsection{Persistence case.}
\begin{thm}\label{thm:periodic}
    If $\bar{\phi} > \bar{\psi}$, the ODE \eqref{sys1} admits a unique $T$-periodic positive solution $R_*(t)$.
\end{thm}
\begin{proof} 
    \textbf{Step 1:} We first prove the existence of a positive $T$-periodic solution. Define  
    \begin{equation}\label{def:F}
    \mathcal{F}(r) = R(T; 0 ,r).
    \end{equation}
    where $R(\cdot;0,r)$ denotes the unique solution of the ODE \eqref{sys1} with initial value $R(0)=r$.

    We begin by constructing an upper bound for $\mathcal{F}$. Since
    \begin{equation}\label{bound1}
        f_0(R(t))\phi(t) - \frac12 \psi(t) \le f_0(R(t))\phi^* - \frac12 \psi_* ,
    \end{equation}
    With $\bar\phi>\bar\psi$, we have
\[
\phi^* \ge \bar\phi > \bar\psi \ge \psi_*,
\]
and thus
\[
0<\frac{\psi_*}{2\phi^*}<\frac12.
\]
By Lemma \ref{lem:f0}, there exists a unique $R_+>0$ such that
\begin{equation}
    \label{upper-sol}
    f_0(R_+)=\frac{\psi_*}{2\phi^*}.
\end{equation}
Evaluating the inequality \eqref{bound1} at $R_+$, it follows that
\begin{equation}
    \label{upper-sol1}
    f_0(R_+)\phi(t)-\frac12\psi(t)
\le
f_0(R_+)\phi^*-\frac12\psi_*
=0,
\end{equation}
so the constant function $R\equiv R_+$ is an upper solution of \eqref{sys1}.
Hence, for every $R_0\in[0,R_+]$, the corresponding solution satisfies
\begin{equation}
    \label{upper-sol2}
    0\le R(t;0,R_0)\le R_+
\qquad \text{for all } t\ge 0,
\end{equation}
where nonnegativity is guaranteed by Theorem \ref{thm:pos}.
In particular,
\begin{equation}
    \label{boundRstar}
    0\le \mathcal F(R_0) = R(T;0,R_0) \le R_+,
\end{equation}
that is, $\mathcal F$ maps $[0,R_+]$ into itself.

Next, we show that $\mathcal F(r)>r$ for all sufficiently small $r>0$.
Since $\bar\phi>\bar\psi$, one may choose $\eta>0$ such that
\[
\left(\frac12-\eta\right)\bar\phi-\frac12\bar\psi>0.
\]
By Lemma \ref{lem:f0}, there exists $\delta>0$ such that
\[
f_0(r)\ge \frac12-\eta
\qquad \text{for all } 0<r\le \delta.
\]
We fix $r_0>0$ so small that
\[
r_0 e^{\frac{\mu}{2}\phi^*T}\le \delta.
\]
Following similar procedures as in \eqref{inter1}--\eqref{inter3}, we have, for all $t\in[0,T]$
\[
R(t;0,r_0)\le r_0 e^{\frac{\mu}{2}\phi^*t}\le r_0 e^{\frac{\mu}{2}\phi^*T} \le \delta.
\]
Therefore,
\[
f_0(R(t;0,r_0))\ge \frac12-\eta
\qquad \text{for all } t\in[0,T].
\]
Consequently,
\[
\begin{aligned}
\ln\frac{\mathcal F(r_0)}{r_0} = \ln\frac{ R(T;0,r_0)}{R(0;0,r_0)} 
&=\mu\int_0^T
\left(f_0(R(t;0,r_0))\phi(t)-\frac12\psi(t)\right)\,dt \\
&\ge
\mu\int_0^T
\left[\left(\frac12-\eta\right)\phi(t)-\frac12\psi(t)\right]\,dt \\
&=
\mu T\left[\left(\frac12-\eta\right)\bar\phi-\frac12\bar\psi\right]
>0.
\end{aligned}
\]
Hence, we obtain
\begin{equation}
    \label{boundr0}
    \mathcal F(r_0)>r_0.
\end{equation}

Since the solution of the ODE \eqref{sys1} depends continuously on the initial value, the mapping $\mathcal{F}$ is continuous. By \eqref{boundRstar} and \eqref{boundr0},  we have
\[
\mathcal F(r_0)-r_0>0,
\qquad
\mathcal F(R_+)-R_+\le 0.
\]
Hence, by the intermediate value theorem, there exists $r_* \in [r_0, R_+]$ such that
\[
\mathcal{F}(r_*) = r_*,
\]
which implies
\[
R(T;0,r_*) = R(0;0,r_*).
\]
Therefore, the solution with initial value $R(0) = r_*$ is a positive $T$-periodic solution, denoted by $R_*(t)$.

\hspace{15pt}

\noindent \textbf{Step 2: } Next, we prove that the $T$-periodic positive solution $R_*(t)$ is unique. 

Assume that $R_{*1}(t)$ and $R_{*2}(t)$ are two different positive $T$-periodic
solutions of \eqref{sys1}. Define
\[
w(t):=\ln\frac{R_{*1}(t)}{R_{*2}(t)}.
\]
Clearly, $w(t)$ is also $T$-periodic. Moreover, taking derivative in $t$ leads to
\[
\begin{aligned}
w'(t)
&=\frac{R_{*1}'(t)}{R_{*1}(t)}-\frac{R_{*2}'(t)}{R_{*2}(t)} \\
&=\mu\phi(t)\bigl(f_0(R_{*1}(t))-f_0(R_{*2}(t))\bigr).
\end{aligned}
\]

Suppose first that there exists $t_* >0$ such that
\[
R_{*1}(t_*)=R_{*2}(t_*).
\]
Since both $R_{*1}$ and $R_{*2}$ solve the same initial value problem at
$t=t_*$, namely
\[
R'=\mu R\left(f_0(R)\phi(t)-\frac12\psi(t)\right),
\qquad
R(t_*)=R_{*1}(t_*)=R_{*2}(t_*),
\]
the uniqueness result in Theorem \ref{thm:pos} implies that
\[
R_{*1}(t)\equiv R_{*2}(t).
\]
Therefore, we should have, either
\[
R_{*1}(t)>R_{*2}(t)\qquad \text{for all } t,
\]
or
\[
R_{*1}(t)<R_{*2}(t)\qquad \text{for all } t.
\]
Without loss of generality, assume that
\[
R_{*1}(t)>R_{*2}(t)\qquad \text{for all } t.
\]
Since $f_0$ is strictly decreasing by Lemma \ref{lem:f0}, we have
\[
f_0(R_{*1}(t))<f_0(R_{*2}(t))
\qquad \text{for all } t.
\]
Using the positivity of $\phi(t)$, it follows that
\[
w'(t)=\mu\phi(t)\bigl(f_0(R_{*1}(t))-f_0(R_{*2}(t))\bigr)<0
\qquad \text{for all } t.
\]
It means that $w$ is strictly decreasing. In particular,
\[
w(t+T)<w(t)\qquad \text{for all } t.
\]
On the other hand, since $w$ is $T$-periodic, we have
\[
w(t+T)=w(t)\qquad \text{for all } t,
\]
which is a contradiction.

Thus, the positive $T$-periodic solution $R_*(t)$ is unique. We complete the proof.
\end{proof}

Finally, we analyze the asymptotic behaviors of the initial value problem \eqref{sys1}--\eqref{sys2}.

\begin{thm}\label{thm:persist}
Assume that $\bar{\phi}>\bar{\psi}$. Let $R_*(t)$ be the unique positive
$T$-periodic solution of \eqref{sys1}. Then $R_*(t)$ is globally
asymptotically stable with respect to positive solutions. That is, for any
$R_0>0$, if $R(t)$ denotes the solution of \eqref{sys1}--\eqref{sys2}, then
\begin{equation}\label{lim}
    \lim\limits_{t\rightarrow \infty}\bigl| R(t) - R_*(t) \bigr| =0.
\end{equation}
\end{thm}

\begin{proof}
From Theorem \ref{thm:periodic}, equation \eqref{sys1} admits a unique positive $T$-periodic solution $R_*(t)$. Let
\[
R_{\min} = \min\limits_{t\in[0,T]}R_*(t),\qquad R_{\max} = \max\limits_{t\in [0,T]}R_*(t).
\]
Since $R_*(t)$ is positive and $T$-periodic, we have
\[
0< R_{\min} \le R_*(t) \le R_{\max} \qquad \text{for all } t \ge 0.
\]

    Similar to the proof of Theorem \ref{thm:periodic}, we define
    \[
    w(t):=\ln \frac{R(t)}{R_*(t)},
    \]
    then a direct calculation implies
    \begin{equation*}
        \begin{split}
            w'(t) &= \frac{R'(t)}{R(t)} - \frac{R_*'(t)}{R_*(t)}\\
            &= \mu \phi(t) \left( f_0(R(t)) - f_0(R_*(t)) \right).
        \end{split}
    \end{equation*}

    If there exists $t_*>0$ such that
    \[
    R(t_*) = R_*(t_*),
    \]
    then, by the uniqueness of solutions to the initial value problem in Theorem \ref{thm:pos}, we have
    \[
    R(t) \equiv R_*(t) \qquad \text{for all }t>t_*,
    \]
    and \eqref{lim} follows immediately.

    Therefore, we may assume that
    \[
    R(t)\neq R_*(t)\qquad \text{for all } t\ge 0.
    \]
    By continuity, either
    \[
    R(t)>R_*(t)\qquad \text{for all } t\ge 0,
    \]
    or
    \[
    R(t)<R_*(t)\qquad \text{for all } t\ge 0.
    \]
    We only consider the first case, since the second one can be treated similarly.

    Assume that 
    \[
    R(t)>R_*(t)\qquad \text{for all } t\ge 0.
    \]
    Then,
    \[
    w(t) > 0, \qquad w'(t) < 0,
    \]
    so $w$ is decreasing and has a limit
    \[
    l := \lim\limits_{t\rightarrow \infty}w(t) \ge 0.
    \]

    We claim that $l=0$. Suppose, on the contrary, that $l>0$. Then, for all
    sufficiently large $t$,
    \[
    w(t) \ge l/2,
    \]
    and hence,
    \[
    R(t) \ge R_*(t) e^{l/2}.
    \]
    Since $f_0$ is strictly decreasing, it follows that
\[
f_0(R(t))-f_0(R_*(t))
\le
f_0\bigl(e^{l/2}R_*(t)\bigr)-f_0(R_*(t)) \qquad \text{for all sufficiently large }t.
\]
Now define
\[
g(x):=f_0(e^{l/2}x)-f_0(x),\qquad x\in [R_{\min},R_{\max}].
\]
Clearly, $g$ is continuous. Moreover, since $e^{l/2}>1$ and $f_0$ is strictly decreasing, 
\[
g(x)<0\qquad \text{for all } x\in [R_{\min},R_{\max}].
\]
By compactness,
there exists $c_0>0$ such that
\[
g(x)\le -c_0\qquad \text{for all } x\in [R_{\min},R_{\max}].
\]
Therefore, for all sufficiently large $t$,
\[
f_0(R(t))-f_0(R_*(t))\le -c_0.
\]
Using the positivity of $\phi(t)$, we obtain
\[
w'(t)
=
\mu\phi(t)\bigl(f_0(R(t))-f_0(R_*(t))\bigr)
\le -\mu\phi_*c_0<0 \qquad \text{for all sufficiently large $t$.}
\]
Integrating this inequality yields
\[
w(t)\to -\infty \qquad \text{as } t\to\infty,
\]
which contradicts the facts that $w(t)>0$ and $w(t)\rightarrow l>0$. Thus, $l=0$.

Therefore, we have
\[
\lim_{t\to\infty}w(t)=0,
\]
that is,
\[
\lim_{t\to\infty}\ln\frac{R(t)}{R_*(t)}=0.
\]
Since $R_*(t)$ is bounded above and below by positive constants, it follows that
\[
\lim_{t\to\infty}|R(t)-R_*(t)|=0.
\]  
\end{proof}

\section{Linear Stability of the Radially Symmetric Periodic Solution}
From the previous section, we know that if $\bar{\phi}>\bar{\psi}$, then
\eqref{seqn1}--\eqref{seqn5} admits a unique radially symmetric positive
$T$-periodic solution $R_*(t)$. Substituting $R_*(t)$ into \eqref{sol:u} and
solving \eqref{eqn2.3}--\eqref{eqn2.4}, we obtain the corresponding periodic
nutrient concentration $u_*(r,t)$ and pressure $p_*(r,t)$, i.e.,
\begin{equation}
    \label{star}
    u_*(r,t) = \frac{\alpha \phi(t)}{I_1(R_*(t))+\alpha I_0(R_*(t))}I_0(r), \qquad p_*(r,t) = \frac{\mu}{4}\psi(t) r^2 - \mu u_*(r,t) + C_1(t),
\end{equation}
where
\[
C_1(t) = \frac{1}{R_*(t)} + \mu \frac{\alpha \phi(t) I_0(R_*(t))}{I_1(R_*(t)) + \alpha I_0(R_*(t))} -\frac{\mu}{4}\psi(t) R_*^2(t).
\]
By \eqref{eqn2.5} and direct computations, we also have
\begin{equation}
    \label{star:p}
    \frac{\p p_*}{\p r}(R_*(t),t) = -R_*'(t).
\end{equation}
\begin{equation}
    \label{star:p2}
    \begin{split}
        \frac{\p^2 p_*}{\p r^2}(R_*(t),t) &= \frac{\mu}2 \psi(t) - \mu \frac{\alpha \phi(t)}{I_1(R_*(t))+\alpha I_0(R_*(t))}\left(I_2(R_*(t))+\frac{1}{R_*(t)}I_1(R_*(t))\right)\\
        &= -\mu\left(f_0(R_*(t))\phi(t) - \frac{\psi(t)}{2}\right) - \mu \frac{\alpha \phi(t) I_2(R_*(t))}{I_1(R_*(t))+\alpha I_0(R_*(t))}\\
        &= -\frac{R_*'(t)}{R_*(t)} - \mu \frac{\alpha \phi(t) I_2(R_*(t))}{I_1(R_*(t))+\alpha I_0(R_*(t))},
    \end{split}
\end{equation}
where we have used the properties of modified Bessel functions listed in Appendix \ref{app1}, together with \eqref{f0} and \eqref{sys1}. These boundary derivative identities will be used below. In what follows,
we denote this radially symmetric $T$-periodic solution by $(u_*(r,t),\,p_*(r,t),\,R_*(t))$. 

The radial dynamics of this periodic solution have already been characterized
in Theorem \ref{thm:persist}. We now turn to its stability with respect to non-radially symmetric perturbations. This is natural from both the modeling and analytical points of view, since perturbations of a free boundary are not necessarily radially symmetric. More precisely, we shall consider the linearized system of \eqref{seqn1}--\eqref{seqn5} around $(u_*(r,t),\,p_*(r,t),\,R_*(t))$ and prove that this periodic solution is
linearly stable when $\mu$ is small.

Assume that the initial conditions are perturbed as follows:
\[
\p \Omega(0): \, r=R_*(0)+\varepsilon \rho_0(\theta), \qquad u\big|_{t=0} = u_*(r,0) + \varepsilon w_0(r,\theta).
\]
We seek a solution in the form
\begin{eqnarray*}
    \displaystyle \p \Omega(t): r=R_*(t) + \varepsilon \rho(\theta,t)+O(\varepsilon^2),\\
    \displaystyle u(r,\theta,t) = u_*(r,t) + \varepsilon w(r,\theta,t) + O(\varepsilon^2),\\
    \displaystyle p(r,\theta,t) = p_*(r,t) + \varepsilon q(r,\theta,t) + O(\varepsilon^2).
\end{eqnarray*}
Substituting these expansions into the system \eqref{seqn1}--\eqref{seqn5}, and using the expansions from \cite{friedman2001nonlinear},
\[
V_n = R_*'(t)+\varepsilon \rho_t (\theta,t) + O(\varepsilon^2),\qquad \kappa = \frac{1}{R_*(t)}-\frac{\varepsilon}{R_*^2(t)}\left(\rho(\theta,t) + \rho_{\theta\theta}(\theta,t)\right)+ O(\varepsilon^2),
\]
we collect the terms of order $\varepsilon$ and obtain the linearized system around $(u_*(r,t),\,p_*(r,t),\,R_*(t))$. The linearized problem takes the following form:
\begin{eqnarray}
    &\displaystyle \label{l1} \Delta w = w \qquad \text{in } B_{R_*(t)} \times \{t>0\},\\
    &\displaystyle \label{l2} \frac{\p w}{\p r} + \alpha w = -\left(\frac{\p^2 u_*}{\p r^2}(R_*(t),t) + \alpha \frac{\p u_*}{\p r}(R_*(t),t) \right) \rho(\theta,t) \qquad \text{on } \p B_{R_*(t)} \times \{t>0\},\\
    &\displaystyle \label{l3} -\Delta q = \mu w \qquad \text{in } B_{R_*(t)} \times \{t>0\},\\
    &\displaystyle \label{l4} q = -\frac{1}{R_*^2(t)}\left(\rho(\theta,t) + \rho_{\theta\theta}(\theta,t)\right) - \frac{\p p_*}{\p r}(R_*(t),t) \rho(\theta,t) \qquad \text{on } \p B_{R_*(t)} \times \{t>0\},\\
    &\displaystyle \label{l5} \rho_t = -\frac{\p^2 p_*}{\p r^2}(R_*(t),t) \rho(\theta,t) - \frac{\p q}{\p r}(R_*(t),\theta,t) \qquad t>0.
\end{eqnarray}
For notational simplicity, define
\begin{equation}
    \beta(t) := \frac{\p^2 u_*}{\p r^2}(R_*(t),t) + \alpha \frac{\p u_*}{\p r}(R_*(t),t).
\end{equation}
To analyze the linearized system, we expand the perturbations in angular Fourier modes. For each integer $n\ge 0$, we seek solutions of the form
\begin{eqnarray*}
    &w(r,\theta,t) = w_n(r,t)\cos(n\theta),\\
    &q(r,\theta,t) = q_n(r,t)\cos(n\theta),\\
    &\rho(\theta,t) = \rho_n(t)\cos(n\theta).
\end{eqnarray*}
Similarly, one may consider solutions of the form
\begin{eqnarray*}
    &w(r,\theta,t) = w_n(r,t)\sin(n\theta),\\
    &q(r,\theta,t) = q_n(r,t)\sin(n\theta),\\
    &\rho(\theta,t) = \rho_n(t)\sin(n\theta),
\end{eqnarray*}
for each integer $n\ge 1$. Since the cosine and sine modes satisfy the same radial equations, it suffices to consider the cosine modes below.

Using the fact $\Delta = \partial_{rr} + \frac{1}{r}\partial_r + \frac{1}{r^2}\partial_{\theta\theta}$, we obtain
\begin{eqnarray}
    &\displaystyle \label{lf1} \frac{\p^2 w_n}{\p r^2}(r,t) + \frac1r \frac{\p w_n}{\p r}(r,t) -\left(\frac{n^2}{r^2} + 1\right) w_n(r,t) = 0 \qquad 0<r<R_*(t), \,t>0,\\
    &\displaystyle \label{lf2} \frac{\p w_n}{\p r}(R_*(t),t) + \alpha w_n(R_*(t),t) = - \beta(t) \rho_n(t) \qquad t>0,\\
    &\displaystyle \label{lf3} \frac{\p^2 q_n}{\p r^2}(r,t) + \frac1r \frac{\p q_n}{\p r}(r,t) - \frac{n^2}{r^2} q_n(r,t) = -\mu w_n(r,t) \qquad 0<r<R_*(t), \,t>0,\\
    &\displaystyle \label{lf4} q_n(R_*(t),t) = \frac{n^2-1}{R_*^2(t)}\rho_n(t) - \frac{\p p_*}{\p r}(R_*(t),t) \rho_n(t) \qquad t>0,\\
    &\displaystyle \label{lf5} \frac{\dif \rho_n(t)}{\dif t} = -\frac{\p^2 p_*}{\p r^2}(R_*(t),t) \rho_n(t) - \frac{\p q_n}{\p r}(R_*(t),t) \qquad t>0.
\end{eqnarray}

We collect some preliminary properties of modified Bessel functions in Appendix \ref{app1}. Using those properties, together with regularity at $r=0$, the solution of \eqref{lf1} has the form
\[
w(r,t) = C_2(t) I_n(r).
\]
Applying the boundary condition \eqref{lf2}, we obtain
\begin{equation}\label{sol:wn}
    w_n(r,t) =-\frac{\beta(t)\rho_n(t)I_n(r)}{I_n'(R_*(t))+\alpha I_n(R_*(t))}.
\end{equation}
To solve for $q_n$, we define $l_n(r,t):=q_n(r,t)+\mu w_n(r,t)$. Since $w_n$ satisfies \eqref{lf1} and $q_n$ satisfies \eqref{lf3}, it follows that $l_n$ satisfies
\[
\frac{\p^2 l_n}{\p r^2}(r,t) + \frac{1}{r}\frac{\p l_n}{\p r}(r,t) - \frac{n^2}{r^2} l_n(r,t) = 0,
\]
The regular solution at $r=0$ is therefore
\[
l_n(r,t)=C_3(t)r^n.
\]
Using the boundary conditions \eqref{lf2} and \eqref{lf4}, together with \eqref{star:p}, we find
\begin{equation}
    \begin{split}
        C_3(t) &= \frac{1}{R_*^n(t)}\left( \frac{n^2-1}{R_*^2(t)}\rho_n(t) - \frac{\p p_*}{\p r}(R_*(t),t)\rho_n(t) + \mu w_n(R_*(t),t)\right)\\
        &= \frac{1}{R_*^n(t)}\left( \frac{n^2-1}{R_*^2(t)}\rho_n(t) + R_*'(t)\rho_n(t) + \mu w_n(R_*(t),t)\right).
    \end{split}
\end{equation}
Thus,
\begin{equation}
    \label{sol:qn}
    q_n(r,t) = \frac{r^n}{R_*^n(t)}\left( \frac{n^2-1}{R_*^2(t)}\rho_n(t) + R_*'(t)\rho_n(t) + \mu w_n(R_*(t),t)\right) - \mu w_n(r,t).
\end{equation}

Substituting \eqref{sol:wn} and \eqref{sol:qn} into \eqref{lf5}, and using the information about the radially-symmetric solution in \eqref{star} -- \eqref{star:p2}, we obtain the governing equation for $\rho_n(t)$:
\begin{equation*}
\begin{split}
    \frac{\dif \rho_n(t)}{\dif t} =& - \frac{n}{R_*(t)}\left(\frac{n^2-1}{R_*^2(t)}\rho_n(t) + R_*'(t)\rho_n(t)+ \mu w_n(R_*(t),t)\right) + \mu \frac{\p w_n}{\p r}(R_*(t),t) + \frac{R_*'(t)}{R_*(t)}\rho_n(t)\\
    &+ \mu \frac{\alpha \phi(t) I_2(R_*(t))}{I_1(R_*(t))+\alpha I_0(R_*(t))} \rho_n(t)\\
    =& -\bigg[\frac{n(n^2-1)}{R_*^3(t)} + \frac{(n-1)R_*'(t)}{R_*(t)} + \frac{\mu \beta(t)I_{n+1}(R_*(t))}{I_n'(R_*(t))+\alpha I_n(R_*(t))} -  \frac{\mu \alpha \phi(t) I_2(R_*(t))}{I_1(R_*(t))+\alpha I_0(R_*(t))}\bigg] \rho_n(t)\\
    =& -\bigg[\frac{n(n^2-1)}{R_*^3(t)} + \frac{(n-1)R_*'(t)}{R_*(t)} +  \mu \bigg(\frac{I_{n+1}(R_*(t))\big( I_1'(R_*(t))+\alpha I_1(R_*(t))\big)}{I_n'(R_*(t))+\alpha I_n(R_*(t))} - I_2(R_*(t)) \bigg)\\ 
    &\frac{\alpha \phi(t)}{I_1(R_*(t))+\alpha I_0(R_*(t))} \bigg] \rho_n(t)\\
    :=& -A_n(t,\mu) \rho_n(t),
\end{split}
\end{equation*}
where
\begin{equation}\label{An}
    \begin{split}
        A_n(t,\mu) =&\; \mu \bigg(\frac{I_{n+1}(R_*(t))\big( I_1'(R_*(t))+\alpha I_1(R_*(t))\big)}{I_n'(R_*(t))+\alpha I_n(R_*(t))} - I_2(R_*(t)) \bigg)\frac{\alpha \phi(t)}{I_1(R_*(t))+\alpha I_0(R_*(t))}\\
        &+ \frac{n(n^2-1)}{R_*^3(t)} + \frac{(n-1)R_*'(t)}{R_*(t)}.
    \end{split}
\end{equation}
Since $\phi(t)$ and $R_*(t)$ are $T$-periodic, and hence $R_*'(t)$ is also
$T$-periodic, it follows that $A_n(t,\mu)$ is $T$-periodic in $t$. Then, $\rho_n(t)$ satisfies a linear differential equation with a periodic coefficient. We shall use the following lemma to study this equation.

\begin{lem}\label{lem:periodic}
    Assume that
    \begin{equation*}
        \frac{\dif \rho}{\dif t} = -A(t)\rho,
    \end{equation*}
    where $A(t)$ is a continuous $T$-periodic function. Then
\begin{equation}\label{sol-f}
    \rho(t)=\rho(0)e^{-\int_0^t A(s)\,\dif s}.
\end{equation}
In particular, 
\begin{itemize}
   \item if $\int_0^T A(t)\,\dif t>0$, then $\rho(t) \rightarrow 0$ exponentially as $t \rightarrow \infty$;
    \item if $\int_0^T A(t)\,\dif t<0$, then $|\rho(t)|\rightarrow \infty$ exponentially as $t\rightarrow \infty$ for every
$\rho(0)\neq0$.
\end{itemize}
\end{lem}

\begin{proof}
    Let $\Gamma := \int_0^T A(t)\dif t$. For $t>0$ large enough, there exists $m\in \mathbb{N}$ such that $t=mT +\tau$ with $\tau \in [0,T)$. Using the $T$-priodicity of $A$, we have
    \[
    \int_0^t A(s)\dif s = \int_0^{mT} A(s)\dif s + \int_{mT}^{mT+\tau} A(s) \dif s = m\Gamma + \int_0^\tau A(s) \dif s.
    \]
    Substituting into the solution formula, we obtain
    \[
    |\rho(t)| = |\rho(0)| e^{-m\Gamma} e^{-\int_0^\tau A(s)\dif s}. 
    \]
    Since $A$ is continuous on $[0,T]$, the quantity $\int_0^\tau A(s)\dif s$ is bounded uniformly for $\tau \in [0,T]$. Hence, there exists $C_4, C_5>0$ such that
    \[
    0< C_4 \le e^{-\int_0^\tau A(s)\dif s} \le  C_5 \qquad \text{for all }\tau \in [0,T].
    \]

If $\Gamma >0$, then
\[
|\rho(t)|\le C_5 |\rho(0)| e^{-m\Gamma} \le C_5|\rho(0)| e^{-(t/T-1)\Gamma}.
\]
Hence, $\rho(t)\rightarrow 0$ exponentially as $t\rightarrow \infty$.

If $\Gamma < 0$ and $\rho(0)\neq 0$, then
\[
|\rho(t)| \ge C_4 |\rho(0)| e^{-m\Gamma} \ge C_4 |\rho(0)| e^{-(t/T-1)\Gamma} \rightarrow \infty,
\]
as $t\rightarrow \infty$.
\end{proof}

We now analyze the stability of each Fourier mode. Since the modes $n=0$ and $n=1$ have special geometric meanings, we treat them separately: the mode $n=0$ corresponds to radial perturbations, while the mode $n=1$
corresponds to translations of the tumor center. The modes $n\ge2$ represent nonradial shape perturbations. Thus we divide our discussion into three cases: (i) $n=0$, (ii) $n=1$, and (iii) $n\ge2$.

\noindent \underline{\bf Case (i) $n=0$.} When $n=0$, $A_0$ defined in \eqref{An} reduced to
\[
A_0(t,\mu) = \mu\bigg( \frac{I_1(R_*(t))\big(I_1'(R_*(t))+\alpha I_1(R_*(t))\big)}{I_0'(R_*(t))+\alpha I_0(R_*(t))} - I_2(R_*(t))\bigg) \frac{\alpha \phi(t)}{I_1(R_*(t))+\alpha I_0(R_*(t))} - \frac{R_*'(t)}{R_*(t)}
\]
By Lemma \ref{lem:periodic}, the long-time behavior of $\rho_0(t)$ is determined
by the sign of the integral of $A_0(t,\mu)$ over one period. Since $R_*(t)$ is
$T$-periodic,
\begin{equation}
    \label{int-Rstar}
    \int_0^T \frac{R_*'(t)}{R_*(t)} \dif t = \ln R_*(T) - \ln R_*(0) = 0.
\end{equation}
Thus, 
\begin{equation}\label{eqn:A0}
    \begin{split}
        \int_0^T A_0(t,\mu)\dif t  &= \int_0^T \mu\bigg( \frac{I_1(R_*(t))\big(I_1'(R_*(t))+\alpha I_1(R_*(t))\big)}{I_0'(R_*(t))+\alpha I_0(R_*(t))} - I_2(R_*(t))\bigg) \frac{\alpha \phi(t)}{I_1(R_*(t))+\alpha I_0(R_*(t))} \dif t\\
        &= \int_0^T \mu\bigg( \frac{I_1'(R_*(t))+\alpha I_1(R_*(t))}{I_0'(R_*(t))+\alpha I_0(R_*(t))} - \frac{I_2(R_*(t))}{I_1(R_*(t))}\bigg) R_*(t)f_0(R_*(t))\phi(t) \dif t,
    \end{split}
\end{equation}
where we used $f_0$ defined in \eqref{f0} for notational simplification. 
Since $R_*(t)>0$, $f_0(R_*(t))>0$, and $\phi(t)>0$, the sign of
$\int_0^T A_0(t,\mu)\dif t$ is determined by the sign of
\begin{equation*}
    \frac{I_1'(R_*(t))+\alpha I_1(R_*(t))}{I_0'(R_*(t))+\alpha I_0(R_*(t))} - \frac{I_2(R_*(t))}{I_1(R_*(t))}.
\end{equation*}

\begin{lem}\label{lem:case0}
    For $n=0$ and any $\mu >0$,
    \begin{equation}
        \label{case0}
        \rho_0(t) \rightarrow 0 \qquad \text{exponentially as } t\rightarrow \infty.
    \end{equation}
\end{lem}
\begin{proof}
    Since $R_*(t)>0$, we have $I_n(R_*(t))>0$ and $I_n'(R_*(t))>0$. Applying Lemma \ref{besselineq2} with $n=1$ gives
    \[
    \frac{I_2(R_*(t))}{I_1'(R_*(t))+\alpha I_1(R_*(t))} < \frac{I_1(R_*(t))}{I_0'(R_*(t))+\alpha I_0(R_*(t))}.
    \]
    Since all denominators are positive, this is equivalent to
    \[
    \frac{I_1'(R_*(t))+\alpha I_1(R_*(t))}{I_0'(R_*(t))+\alpha I_0(R_*(t))} > \frac{I_2(R_*(t))}{I_1(R_*(t))}.
    \]
    Therefore, the integrand in \eqref{eqn:A0} is positive for all $\mu>0$ and $t \in [0,T]$. Hence
    \[
    \int_0^T A_0(t,\mu)\dif t > 0 
    \]
    By Lemma \ref{lem:periodic}, it follows that $\rho_0(t)\rightarrow 0$ exponentially as $t\rightarrow \infty$.

\end{proof}

\begin{remark}
    For $n=0$, 
    \[
    r= R_*(t) + \varepsilon \rho_0(t)\cos(0) = R_*(t) + \varepsilon \rho_0(t),
    \]
    which indicates that the perturbation is a radial perturbation. The result in Lemma \ref{lem:case0} is just another indication the results in Theorem \ref{thm:persist}.
\end{remark}

\noindent \underline{\bf Case (ii) $n=1$.} 
\begin{lem}\label{lem:case1}
    For $n=1$ and any $\mu>0$, we have
    \begin{equation}
        \label{case1}
         \rho_1(t) = \rho_1(0).
    \end{equation}
\end{lem}
\begin{proof}
    It follows from \eqref{An} that
    \begin{equation*}
    \begin{split}
        A_1(t,\mu) = \mu \Big(I_2(R_*(t))- I_2(R_*(t))\Big)\frac{\alpha \phi(t)}{I_1(R_*(t))+\alpha I_0(R_*(t))} = 0.
    \end{split}
\end{equation*}
Thus, the governing equation for $\rho_1(t)$ reduces to 
\[
\frac{\dif \rho_1(t)}{\dif t} = 0,
\]
which immediately implies \eqref{case1}. 
\end{proof}

\begin{remark}
    For $n=1$,
    \[
    r= R_*(t) + \varepsilon \rho_1(t) \cos(\theta) = R_*(t) + \varepsilon \rho_1(0) \cos(\theta).
    \]
    This corresponds to a translation of the tumor center, and the magnitude
of the $n=1$ perturbation is invariant in time.
\end{remark}

\noindent \underline{\bf Case (iii) $n\ge 2$.} By Lemma \ref{lem:periodic}, the stability of the $n$-th Fourier mode is
determined by the sign of
\begin{equation}
    \label{eqn:G}
    \Gamma_n(\mu):=\int_0^T A_n(t,\mu)\,\dif t.
\end{equation}
Unlike the other two cases, where the sign of $\Gamma_0(\mu)$ and $\Gamma_1(\mu)$ can be determined, here we need to solve for $\mu$ from $\Gamma_n(\mu)=0$. It should be noted that $A_n(t,\mu)$ defined in \eqref{An} depends also on the $T$-periodic solution $R_*(t)$, and $R_*(t)$ depends on $\mu$. So, to be more accurate, we write \eqref{eqn:G} as
\begin{equation}
    \label{eqn:G:a}
    \Gamma_n(\mu):=\int_0^T A_n(t,\mu, R_*(t;\mu))\,\dif t.
\end{equation}
Using the periodicity of $R_*(t)$ in \eqref{int-Rstar}, we have
\begin{equation}\label{comp:G}
    \Gamma_n(\mu) = \int_0^T \frac{n(n^2-1)}{R_*^3(t;\mu)} \dif t - \mu \int_0^T H_n(t,R_*(t;\mu))\dif t,
\end{equation}
where
\begin{equation}\label{comp:H}
    H_n(t,R):= \frac{\alpha \phi(t)}{I_1(R)+\alpha I_0(R)} \left(I_2(R) -  \frac{I_{n+1}(R)\big(I_1'(R)+\alpha I_1(R)\big)}{I_n'(R)+\alpha I_n(R)}\right).
\end{equation}
By Lemma \ref{besselineq2}, the sequence
    \[
    \frac{I_{n+1}(x)}{I_n'(x)+\alpha I_n(x)}
    \]
    is strictly decreasing in $n$ when $x>0$. Hence, for $n\ge2$ and $R>0$,
\[
\frac{I_{n+1}(R)}
{I_n'(R)+\alpha I_n(R)}
<
\frac{I_2(R)}
{I_1'(R)+\alpha I_1(R)}.
\]
Since
\[
I_1'(R)+\alpha I_1(R)>0,
\]
it follows that
\[
I_2(R)
-
\frac{
I_{n+1}(R)
\bigl(I_1'(R)+\alpha I_1(R)\bigr)}
{I_n'(R)+\alpha I_n(R)}
>0.
\]
Moreover,
\[
\frac{\alpha\phi(t)}
{I_1(R)+\alpha I_0(R)}>0.
\]
Therefore,
\[
H_n(t,R) > 0 \qquad \text{for all }t\in[0,T] \text{ and } R>0.
\]

In order to remove the dependence of $R_*(t;\mu)$ on $\mu$, we impose the
stronger assumption
\begin{equation}
    \label{strong-cond}
    \phi_* > \psi^*.
\end{equation}
Then, by \eqref{phi} and \eqref{psi},
\[
\psi_* \le  \bar{\psi} \le \psi^* < \phi_* \le \bar{\phi} \le  \phi^*.
\]
In addition to $R_+$ defined in \eqref{upper-sol}, we also define 
\[
R_- := f_0^{-1}\left(\frac{\psi^*}{2\phi_*}\right).
\]
By a comparison argument, analogous to that used in \eqref{upper-sol} -- \eqref{upper-sol2}, it can be proved that
\[
0< R_- \le R_*(t;\mu) \le R_+ \qquad \text{for all $t\in[0,T]$ and $\mu>0$}.
\]
Since $H_n(t,R)$ is continuous and strictly positive on the compact set $[0,T]\times[R_-,R_+]$, there exist constants $H_n^-, H_n^+>0$ such that
\begin{equation}
    \label{bound:H}
    0< H_n^- \le H_n(t,R_*(t;\mu) )\le H_n^+ \qquad \text{for all $t\in[0,T]$ and $\mu>0$}.
\end{equation}
Thus,
\begin{equation}
    \label{lb}
    \Gamma_n(\mu) \ge \left(\frac{n(n^2-1)}{R_+^3} - \mu H_n^+ \right)T > 0 \qquad \text{for sufficiently small $\mu>0$},
\end{equation}
and
\begin{equation}
    \label{ub}
    \Gamma_n(\mu) \le \left(\frac{n(n^2-1)}{R_-^3} - \mu H_n^- \right)T < 0 \qquad \text{for sufficiently large $\mu>0$}.
\end{equation}
Since $\Gamma_n(\mu)$ is continuous in $\mu$, the intermediate value theorem implies that there exists
at least one root $\mu_n>0$ such that
\[
\Gamma_n(\mu_n) = 0.
\]

Let
\begin{equation}
    \label{mu:n:s}
    \mu_n^*:=\inf\{\mu>0:\Gamma_n(\mu)=0\}.
\end{equation}
Since \(\Gamma_n(\mu)>0\) for sufficiently small \(\mu>0\), we have
\(\mu_n^*>0\). Moreover, we claim that $\mu_n^* \rightarrow \infty$ as $n\rightarrow \infty$. Indeed, using \eqref{besseln}, we have
\[
H_n(t,R) \le \frac{\alpha \phi(t)}{I_1(R)+\alpha I_0(R)} I_2(R) \le \alpha \phi(t) \frac{I_2(R)}{I_1(R)} \le \frac12 \alpha \phi(t) R,
\]
hence
\[
H_n(t,R_*(t;\mu)) \le \frac12 \alpha \phi(t) R_+.
\]
If $\mu_n$ is a positive root of $\Gamma_n(\mu)=0$, then
\[
0 = \Gamma_n(\mu_n) \ge \left(\frac{n(n^2-1)}{R_+^3} - \mu_n \frac{1}{2}\alpha \bar{\phi} R_+\right) T,
\]
which implies
\[
\mu_n \ge \frac{2n(n^2-1)}{\alpha \bar{\phi}R_+^4}.
\]
Taking the infimum over all positive roots gives
\begin{equation}
    \label{critical:ub}
    \mu_n^* \ge \frac{2n(n^2-1)}{\alpha \bar{\phi}R_+^4} \rightarrow \infty
\end{equation}
as $n\rightarrow \infty$.

Since Lemmas \ref{lem:case0} and \ref{lem:case1} hold for all $\mu>0$, we set $\mu_0^*=\mu_1^*=\infty$. Define
\begin{equation}
    \label{mu:s}
    \mu_* = \min\{\mu_0^*, \mu_1^*, \mu_2^*, \mu_3^*, \mu_4^*,\cdots\}.
\end{equation}
Since $\mu_n^*\rightarrow \infty$, the minimum is attained among finitely many
indices. Hence
\[
0<\mu_*<\infty.
\]
If $0<\mu < \mu_*$, then by the definitions of $\mu_n^*$ and $\mu_*$ in \eqref{mu:n:s} and \eqref{mu:s}, we obtain 
\[
\Gamma_n(\mu) = \int_0^T A_n(t,\mu) \dif t  > 0 \qquad \text{for all }n\ge 2.
\]
Together with the cases when $n=0$ and $n=1$, we know that all Fourier modes except the translation mode $n=1$ decay exponentially as $t\rightarrow \infty$. Thus, the $T$-periodic radially symmetric solution $(u_*(r,t),p_*(r,t),R_*(t))$ is linearly stable.

\begin{remark}\label{rm:bound}
    Based on \eqref{critical:ub}, we may obtain an explicit lower bound, $\mu_* \ge \frac{12}{\alpha \bar{\phi} R_+^4}$. Consequently, a sufficient condition for linear stability of the positive
$T$-periodic solution is
\[
0<\mu<\frac{12}{\alpha\bar{\phi}R_+^4}.
\]
This condition is sufficient but not necessary.
\end{remark}

\section{Numerical Simulations}

In this section, we perform some numerical simulations to illustrate and support our theoretical findings. For the $T$-periodic coefficients $\phi(t)$ and $\psi(t)$, we choose
\begin{equation}
    \label{num:phi}
    \phi(t) = \phi_0 \big(1+a_{\phi} \sin(2\pi t/T)\big),
\end{equation}
\begin{equation}
    \label{num:psi}
    \psi(t) = \psi_0 \big(1+a_{\psi} \sin(2\pi t/T + \theta)\big),
\end{equation}
where $\phi_0,\psi_0 > 0$ and $|a_{\phi}|, |a_{\psi}| < 1$.
Clearly, both functions are continuous, strictly positive, and $T$-periodic. Moreover, a direct calculation gives
\begin{equation}
    \label{ave}
    \bar{\phi} = \phi_0, \qquad \bar{\psi} = \psi_0.
\end{equation}
Thus, by choosing $\phi_0$ and $\psi_0$, we can control the averaged nutrient supply and demand. The phase shift $\theta$ in \eqref{num:psi} allows the cycles of nutrient supply and demand to be nonsynchronized.

\subsection{Radially symmetric case}

We take
\[
a_\phi=0.5,\qquad a_\psi=0.3,\qquad \theta=\frac{\pi}{12},\qquad
T=1,\qquad \mu=3,\qquad \alpha=1.
\]
With these choices, we solve the initial value problem \eqref{sys1}--\eqref{sys2}
on the time interval $[0,50]$.

In Figure \ref{fig:vanish1}, we choose $\phi_0=3$ and $\psi_0=5$, so that $\bar{\phi} < \bar{\psi}$. In Figure \ref{fig:vanish2}, we choose $\phi_0=5$ and $\psi_0=5$, so that $\bar{\phi}=\bar{\psi}$. In both cases, the solution starts from $R_0=1.2$ and gradually converges to
zero, which is consistent with the vanishing result in Theorem \ref{dich}.

\begin{figure}[H]
\centering
\begin{subfigure}{.48\textwidth}
\centering
\includegraphics[width=\linewidth]{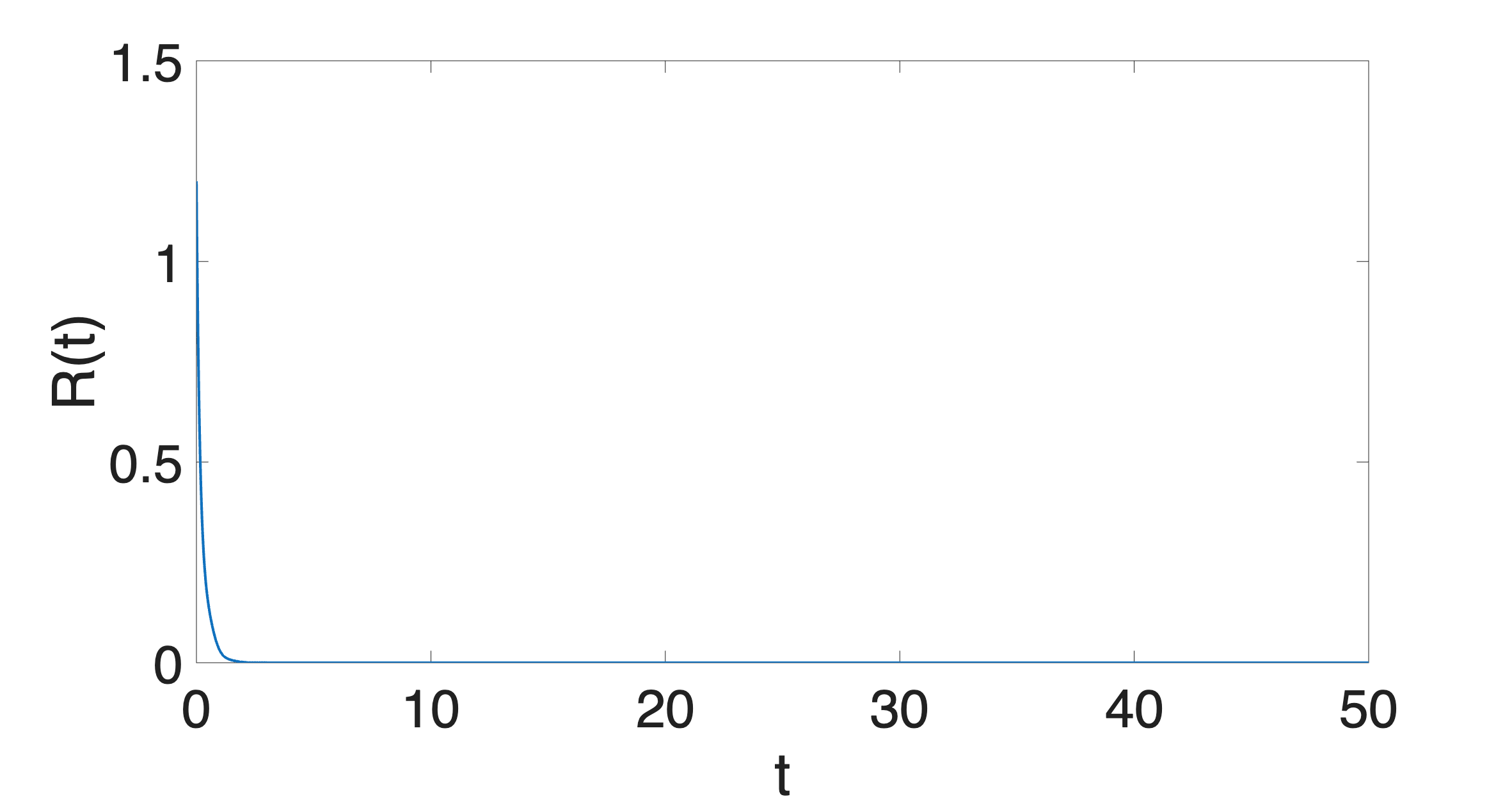}
\caption{$\phi_0=3$, $\psi_0=5$, and $R_0=1.2$.}
\label{fig:vanish1}
\end{subfigure}
\hfill
\begin{subfigure}{.48\textwidth}
\centering
\includegraphics[width=\linewidth]{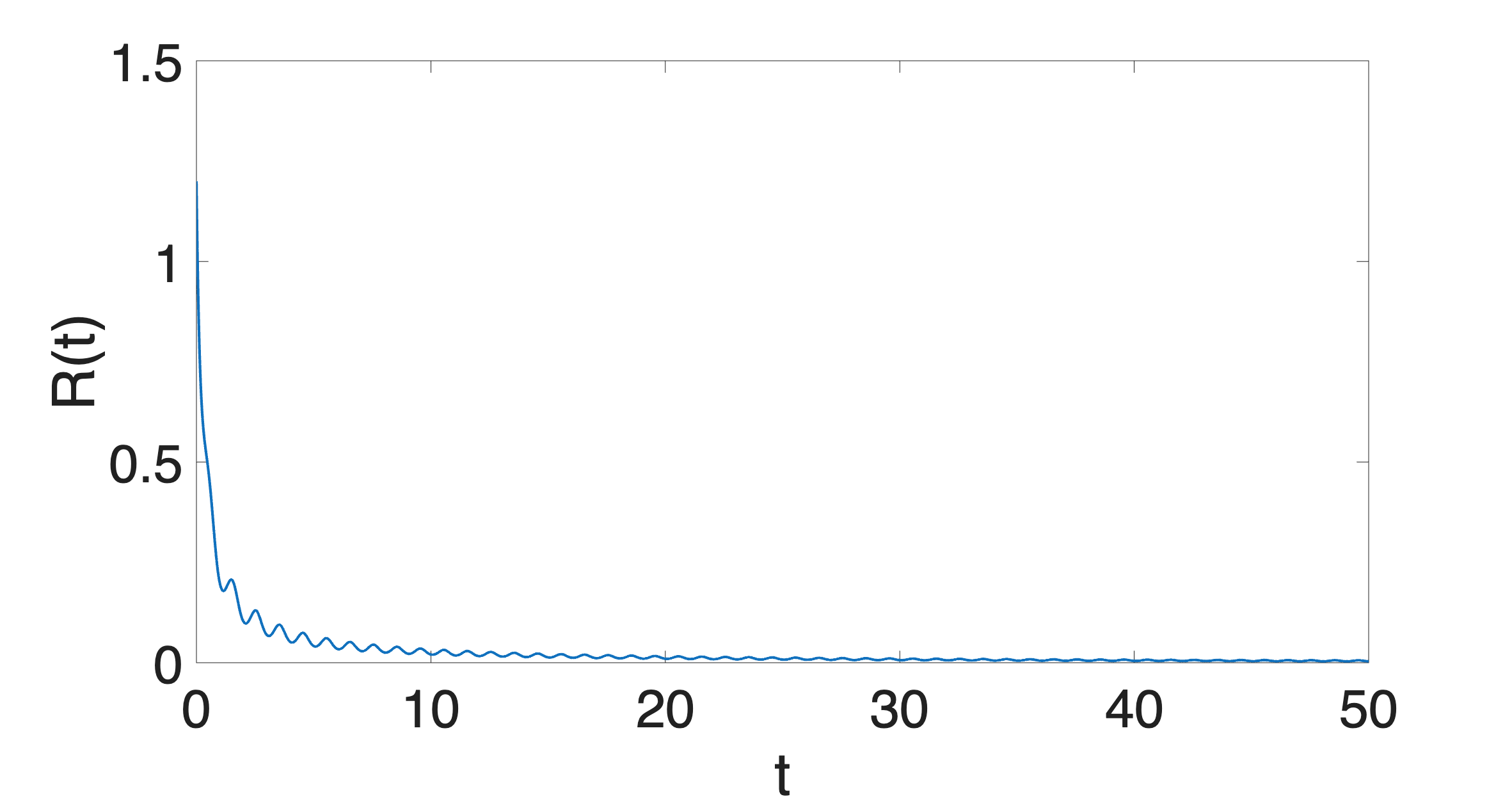}
\caption{$\phi_0=5$, $\psi_0=5$, and $R_0=1.2$.}
\label{fig:vanish2}
\end{subfigure}
\caption{Vanishing dynamics of the tumor radius when $\bar{\phi}\le \bar{\psi}$.}
\label{fig:vanishing}
\end{figure}

In Figures \ref{fig:persist1} and \ref{fig:persist2}, we choose $\phi_0=5$ and $\psi_0 = 3$, so that $\bar{\phi}>\bar{\psi}$. In Figure \ref{fig:persist1}, we take $R_0=1.2$, while in Figure \ref{fig:persist2}, we take $R_0=0.5$. In both cases, the solutions converge to the same positive periodic solution, illustrating the persistence result in Theorem \ref{dich}.

\begin{figure}[H]
\centering
\begin{subfigure}{.48\textwidth}
\centering
\includegraphics[width=\linewidth]{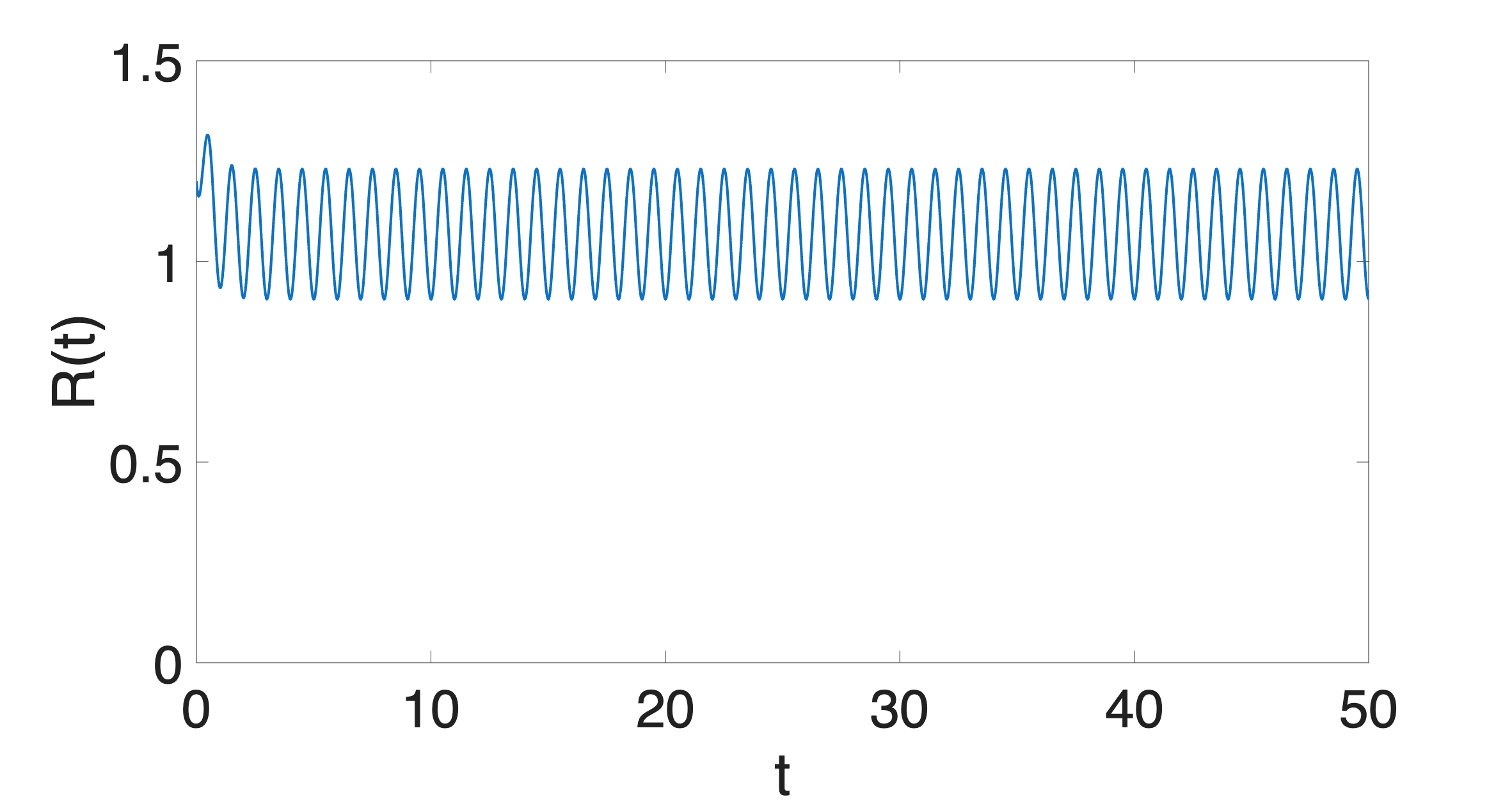}
\caption{$\phi_0=5$, $\psi_0=3$, and $R_0=1.2$.}
\label{fig:persist1}
\end{subfigure}
\hfill
\begin{subfigure}{.48\textwidth}
\centering
\includegraphics[width=\linewidth]{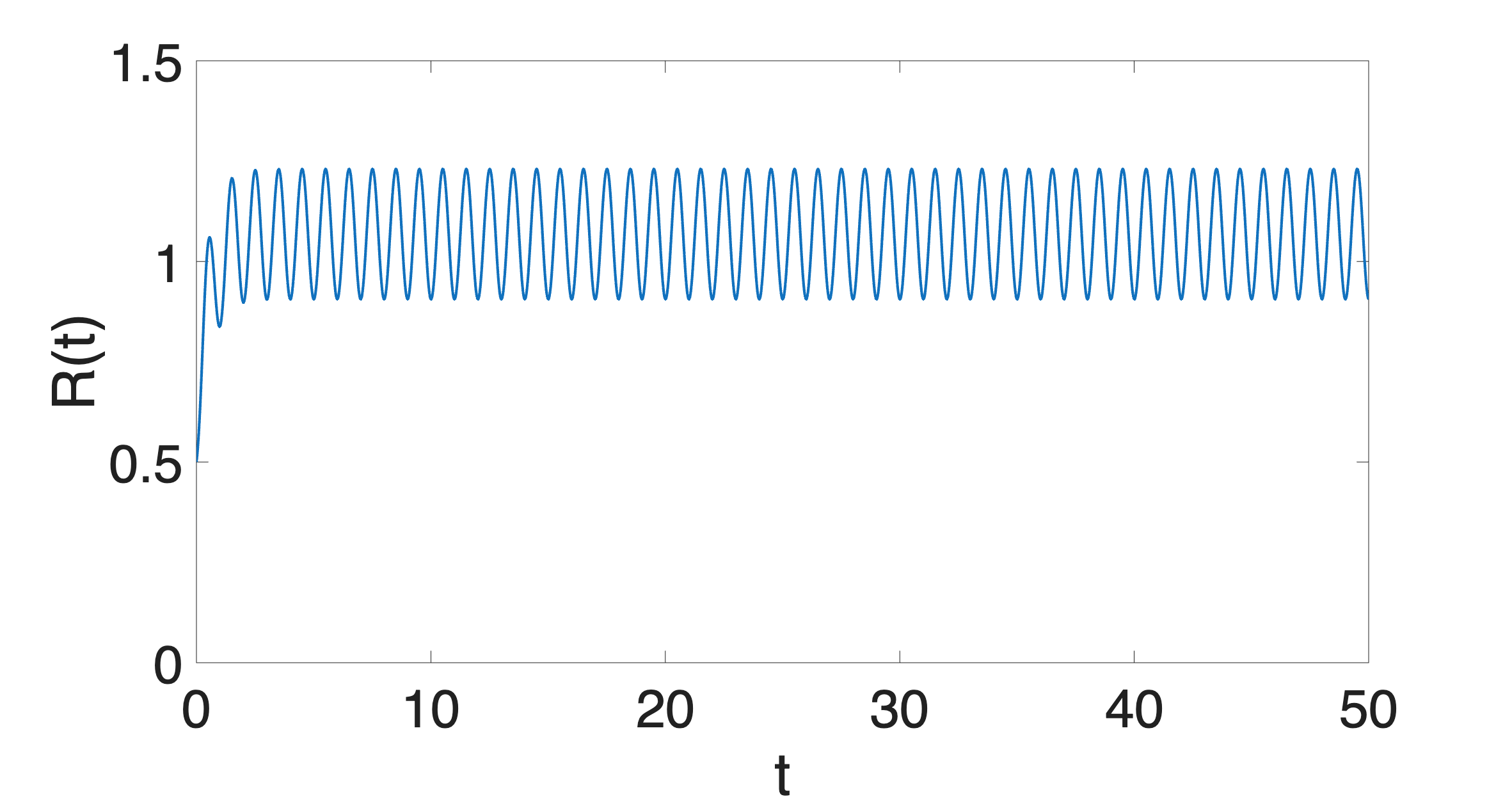}
\caption{$\phi_0=5$, $\psi_0=3$, and $R_0=0.5$.}
\label{fig:persist2}
\end{subfigure}
\caption{Persistence dynamics of the tumor radius when $\bar{\phi}> \bar{\psi}$.}
\label{fig:persistence}
\end{figure}

\subsection{Linear stability}
To illustrate the linear stability result, we take
\[
a_\phi=0.2,\qquad a_\psi=0.3,\qquad \phi_0=5,\qquad \psi_0=3, \qquad \theta=\frac{\pi}{12},\qquad
T=1,\qquad \alpha=1.
\]
By \eqref{num:phi} and \eqref{num:psi}, we have
\[
\phi_* = \phi_0(1-a_\phi) = 4.0,\qquad \psi^*=\psi_0(1+a_\psi) =3.9.
\]
Thus, the stronger condition $\phi_* > \psi^*$ is satisfied. Moreover,
\[
\phi^*=\phi_0(1+a_\phi)=6.0,\qquad
\psi_*=\psi_0(1-a_\psi)=2.1.
\]
Therefore
\[
R_+ = f_0^{-1}\left(\frac{\psi_*}{2\phi^*}\right)
      = f_0^{-1}\left(\frac{2.1}{12}\right)
      \approx 2.4707.
\]
By Remark \ref{rm:bound}, the radially symmetric $T$-periodic solution is linearly stable whenever
\[
0<\mu<\frac{12}{\alpha\bar\phi R_+^4}\approx 0.0644.
\]

For a fixed value of $\mu$, we solve the initial value problem \eqref{sys1} -- \eqref{sys2} with initial value $R_0=1.2$ on the time interval $[0,100]$. Since the solution converges to the positive periodic solution, we
approximate $R_*(t)$ by the solution over the last period, i.e.,
\[
R_*(t) \approx R(100-T+t) \qquad t\in[0,T].
\]
We then compute
\[
\Gamma_n(\mu) = \int_0^T A_n(t,\mu,R_*(t)) \dif t
\]
for the first several Fourier modes. According to Lemma \ref{lem:periodic}, the $n$-th mode decays if $\Gamma_n(\mu)>0$ and grows if $\Gamma_n(\mu)<0$. For computational purposes, we use \eqref{bessel5} in \eqref{An}, together with the periodicity identity \eqref{int-Rstar}, we write
{\small \begin{equation*}
    \begin{split}
        \Gamma_n(\mu) =& \mu \int_0^T \frac{\alpha \phi(t)}{I_1(R_*(t))+\alpha I_0(R_*(t))}\left(\frac{I_{n+1}(R_*(t))\Big(I_2(R_*(t))+\big(\frac{1}{R_*(t)}+\alpha\big) I_1(R_*(t))\Big)}{I_{n+1}(R_*(t))+ \big(\frac{n}{R_*(t)}+\alpha\big) I_n(R_*(t))}- I_2(R_*(t)) \right) \dif t\\
        &+  \int_0^T \frac{n(n^2-1)}{R_*^3(t)}\dif t.
    \end{split}
\end{equation*}}

When $\mu=0.05 < 0.0644$, the sufficient condition in Remark \ref{rm:bound} guarantees the linear stability of the periodic solution. The computed values
of $\Gamma_n(\mu)$ for $0\le n\le 13$ are listed in Table \ref{table1}.

\begin{table}[H]
\centering
\begin{tabular}{ |c|c|c|c|c|c|c|c| } 
 \hline
 mode $n$ & 0 & 1 & 2 & 3 & 4 & 5 & 6   \\ 
 \hline
 $\Gamma_n(\mu)$ & 0.0356 & 0 & 4.9301 & 19.7474 & 49.3883 & 98.7921 & 172.8987  \\
 \hline
 mode $n$ & 7 & 8 & 9 & 10 & 11 & 12 & 13 \\
 \hline
 $\Gamma_n(\mu)$ & 276.6486 & 414.9819 & 592.8393 & 815.1612 & 1.0869e+03 & 1.4130e+03 & 1.7983e+03 \\
 \hline
\end{tabular}
\caption{Computed values of $\Gamma_n(\mu)$ for $\mu=0.05$.}
\label{table1}
\end{table}

The value $\Gamma_1(\mu)=0$ corresponds to the translation mode. All computed non-translation modes satisfy $\Gamma_n(\mu)>0$, which is
consistent with the linear stability predicted by the sufficient condition.

We also consider a much larger value of the tumor aggressiveness parameter,
namely $\mu=30$. The computed values of $\Gamma_n(\mu)$ for $0\le n\le 13$
are listed in Table \ref{table2}.

\begin{table}[H]
\centering
\begin{tabular}{ |c|c|c|c|c|c|c|c| } 
 \hline
 mode $n$ & 0 & 1 & 2 & 3 & 4 & 5 & 6   \\ 
 \hline
 $\Gamma_n(\mu)$ & 21.1061 & 1.5545e-16 & -0.8156 & 12.8696 & 44.0274 & 97.3292 & 177.9004  \\
 \hline
 mode $n$ & 7 & 8 & 9 & 10 & 11 & 12 & 13 \\
 \hline
 $\Gamma_n(\mu)$ & 291.0200 & 442.0283 & 636.2931 & 879.1958 & 1.1761e+03 & 1.5325e+03 & 1.9536e+03 \\
 \hline
\end{tabular}
\caption{Computed values of $\Gamma_n(\mu)$ for $\mu=30$.}
\label{table2}
\end{table}

The value of $\Gamma_1(\mu)$ is approximately zero; the small nonzero value is
due to numerical error. Since $\Gamma_2(\mu)<0$, the second Fourier mode grows
exponentially by Lemma \ref{lem:periodic}. Hence, the radially symmetric
periodic solution is unstable when $\mu=30$.

Biologically, the parameter $\mu$ measures the intensity of tumor expansion due
to cell proliferation, and therefore may be interpreted as a tumor
aggressiveness parameter. When $\mu$ is small, the radially symmetric
$T$-periodic solution is linearly stable, which means that small perturbations
of the tumor boundary decay in time, and the tumor maintains an approximately spherical periodic growth
pattern. When $\mu$ is large, however, some nonradial modes may grow, and the
radially symmetric $T$-periodic solution loses linear stability. This indicates
the onset of shape instability, suggesting that a sufficiently aggressive tumor
may develop non-spherical boundary deformations. Similar relationships between
the loss of linear stability and the tumor aggressiveness parameter $\mu$ have
also been observed in other free boundary tumor models
\cite{bazaliy2003global,friedman2006asymptotic,friedman2006bifurcation}.

\bigskip







\appendix
\section{Properties of Bessel Functions}\label{app1}

Recall that the modified Bessel function of the first kind $I_n(x)$ 
satisfies the differential equation
\begin{equation}\label{bessel1}
    I''_n(x) + \frac{1}{x}I'_n(x) - \bigg(1+\frac{n^2}{x^2}\bigg) I_n(x) = 0\hspace{2em}x>0,\; n\ge 0,
\end{equation}
and is given by
\begin{equation}
    \label{bessel2}
    I_n(x) = \bigg(\frac{x}{2}\bigg)^n \sum_{k=0}^\infty \frac{1}{k!\Gamma(n+k+1)}\bigg(\frac{x}{2}\bigg)^{2k},
\end{equation}
from which it is easy to derive
\begin{equation}
    I_n(x) > 0 \;\text{ and }\; I_n'(x) > 0  \hspace{2em} \mbox{for }x>0, \; n\ge 0,
\end{equation}
\begin{equation}
    \label{besseln}
\frac{I_{n+1}(x)}{I_n(x)} < \frac{x}{2n} \hspace{2em} \mbox{for } x>0, \; n\ge 1,
\end{equation}
Furthermore, $I_n(\xi)$ satisfies the following properties:
\begin{gather}
I'_n(x) + \frac{n}{x}I_n(x) = I_{n-1}(x) \hspace{2em} n\ge 1, \label{bessel3}\\
I'_n(x) - \frac{n}{x}I_n(x) = I_{n+1}(x) \hspace{2em} n\ge 0, \label{bessel4}\\
x^{n+1} I_n(x) = \frac{\p}{\p x}(x^{n+1} I_{n+1}(x)) \hspace{2em} n\ge 0,\label{bessel5}\\
I_{n-1}(x) - I_{n+1}(x) = \frac{2n}{x}I_n(x) \hspace{2em} n\ge 1, \label{bessel6}\\
\frac{I_n(x)}{x} \text{ is increasing in $x$ for } x>0 \hspace{2em} n\ge 1.\label{bessel7}
\end{gather}

\section{Inequalities of Bessel Functions}
\begin{lem}\label{besselineq1}
    For $n \ge 1$,
    \[
    I_n^2(x) > I_{n+1}(x) I_{n-1}(x) \qquad \mbox{for }x>0.
    \]
\end{lem}
\begin{proof}
    Using the series representation of the modified Bessel function \eqref{bessel2}, we have
    \begin{equation*}
    \begin{split}
        I_n^2(x) &= \bigg(\sum_{i=0}^\infty \frac{\big(\frac{x}{2}\big)^{2i+n}}{i! \Gamma(n+i+1)} \bigg)\bigg(\sum_{j=0}^\infty \frac{\big(\frac{x}{2}\big)^{2j+n}}{j! \Gamma(n+j+1)} \bigg) \\
        &= \sum_{k=0}^\infty \frac{\Gamma(2n+2k+1)}{k! \Gamma(n+k+1)^2 \Gamma(2n+k+1)}\Big(\frac{x}{2}\Big)^{2k+2n},
        \end{split}
    \end{equation*}
    and
    \begin{equation*}
    \begin{split}
        I_{n+1}(x)I_{n-1}(x) &= \bigg(\sum_{i=0}^\infty \frac{\big(\frac{x}{2}\big)^{2i+n+1}}{i! \Gamma(n+i+2)} \bigg)\bigg(\sum_{j=0}^\infty \frac{\big(\frac{x}{2}\big)^{2j+n-1}}{j! \Gamma(n+j)} \bigg) \\
        &= \sum_{k=0}^\infty \frac{\Gamma(2n+2k+1)}{k! \Gamma(n+k)\Gamma(n+k+2)\Gamma(2n+k+1)}\Big(\frac{x}{2}\Big)^{2k+2n},
        \end{split}
    \end{equation*}
    Thus,
    \begin{equation*}
        \begin{split}
            &I_n^2(x) - I_{n+1}(x)I_{n-1}(x)\\
            &\qquad = \sum_{k=0}^\infty \frac{\Gamma(2n+2k+1)}{k! \Gamma(2n+k+1)} \left( \frac{1}{\Gamma(n+k+1)^2} - \frac{1}{\Gamma(n+k)\Gamma(n+k+2)}\right)\Big(\frac{x}{2}\Big)^{2k+2n}.
        \end{split}
    \end{equation*}
    To prove $I_n^2(x)-I_{n+1}(x)I_{n-1}(x) >0$, we only need to show that each coefficient is positive.  Therefore, it is sufficient to prove
    \[
    \frac{1}{\Gamma(n+k+1)^2} - \frac{1}{\Gamma(n+k)\Gamma(n+k+2)} >0.
    \]
    Using
\[
\Gamma(n+k+1)=(n+k)\Gamma(n+k),
\]
and
\[
\Gamma(n+k+2)=(n+k+1)\Gamma(n+k+1),
\]
we find
\begin{equation*}
    \begin{split}
        \frac{1}{\Gamma(n+k+1)^2} - \frac{1}{\Gamma(n+k)\Gamma(n+k+2)} = \frac{1}{\Gamma(n+k+1)\Gamma(n+k)}\left(\frac{1}{n+k} - \frac{1}{n+k+1}\right) > 0.
    \end{split}
\end{equation*}
Therefore, the proof is complete.
\end{proof}

\begin{lem}\label{besselineq2}
    For $n \ge 1$ and $\alpha >0$,
    \[
    \frac{I_{n+1}(x)}{I_n'(x)+\alpha I_n(x)} < \frac{I_n(x)}{I_{n-1}'(x)+ \alpha I_{n-1}(x)} \qquad \mbox{for }x>0.
    \]
\end{lem}
\begin{proof}
    Since $I_n(x)>0$ and $I_n'(x)>0$ for $x>0$, all denominators are positive.
    Thus it suffices to prove
    \[
    I_{n+1}(x)\Big(I_{n-1}'(x) + \alpha I_{n-1}(x)\Big) < I_n(x) \Big(I_n'(x) + \alpha I_n(x)\Big).
    \]
    By \eqref{bessel4}, this inequality is equivalent to
    \[
    I_{n+1}(x)\Big(I_n(x) + \frac{n-1}{x}I_{n-1}(x) + \alpha I_{n-1}(x) \Big) < I_n(x) \Big( I_{n+1}(x) + \frac{n}{x}I_n(x) + \alpha I_n(x)\Big).
    \]
    After canceling the common term $I_{n+1}(x) I_n(x)$, it remains to prove
    \[
    \Big(\frac{n-1}{x}+\alpha\Big) I_{n+1}(x) I_{n-1}(x) < \Big(\frac{n}{x} + \alpha \Big) I_n^2(x).
    \]
    By Lemma \ref{besselineq1}, we have
    \[
    \Big(\frac{n}{x} + \alpha \Big) I_n^2(x) > \Big(\frac{n}{x} + \alpha \Big) I_{n+1}(x)I_{n-1}(x) > \Big(\frac{n-1}{x}+\alpha\Big) I_{n+1}(x) I_{n-1}(x),
    \]
    and the desired inequality follows.
\end{proof}

    






\bigskip
\begin{thebibliography}{10}

\bibitem{bai2013qualitative}
M.~Bai and S.~Xu.
\newblock Qualitative analysis of a mathematical model for tumor growth with a periodic supply of external nutrients.
\newblock {\em Pacific Journal of Applied Mathematics}, 5(4):217, 2013.

\bibitem{bazaliy2003global}
B.~Bazaliy and A.~Friedman.
\newblock Global existence and asymptotic stability for an elliptic-parabolic free boundary problem: an application to a model of tumor growth.
\newblock {\em Indiana University Mathematics Journal}, pages 1265--1304, 2003.

\bibitem{bazaliy2003free}
B.~V. Bazaliy and A.~Friedman.
\newblock A free boundary problem for an elliptic-parabolic system: application to a model of tumor growth.
\newblock {\em Communications in Partial Differential Equations}, 28(3-4):517--560, 2003.

\bibitem{cristini2003nonlinear}
V.~Cristini, J.~Lowengrub, and Q.~Nie.
\newblock Nonlinear simulation of tumor growth.
\newblock {\em Journal of Mathematical Biology}, 46(3):191--224, 2003.

\bibitem{cui2007bifurcation}
S.~Cui and J.~Escher.
\newblock Bifurcation analysis of an elliptic free boundary problem modelling the growth of avascular tumors.
\newblock {\em SIAM Journal on Mathematical Analysis}, 39(1):210--235, 2007.

\bibitem{cui2009well}
S.~Cui and J.~Escher.
\newblock Well-posedness and stability of a multi-dimensional tumor growth model.
\newblock {\em Archive for Rational Mechanics and Analysis}, 191:173--193, 2009.

\bibitem{cui2001analysis}
S.~Cui and A.~Friedman.
\newblock Analysis of a mathematical model of the growth of necrotic tumors.
\newblock {\em Journal of Mathematical Analysis and Applications}, 255(2):636--677, 2001.

\bibitem{FFH1}
M.~Fontelos, A.~Friedman, and B.~Hu.
\newblock Mathematical analysis of a model for the initiation of angiogenesis.
\newblock {\em SIAM Journal on Mathematical Analysis}, 33:1330--1355, 2002.

\bibitem{fontelos2003symmetry}
M.~A. Fontelos and A.~Friedman.
\newblock Symmetry-breaking bifurcations of free boundary problems in three dimensions.
\newblock {\em Asymptotic Analysis}, 35(3-4):187--206, 2003.

\bibitem{friedman2007mathematical}
A.~Friedman.
\newblock Mathematical analysis and challenges arising from models of tumor growth.
\newblock {\em Mathematical Models and Methods in Applied Sciences}, 17(supp01):1751--1772, 2007.

\bibitem{friedman2006asymptotic}
A.~Friedman and B.~Hu.
\newblock Asymptotic stability for a free boundary problem arising in a tumor model.
\newblock {\em Journal of Differential Equations}, 227(2):598--639, 2006.

\bibitem{friedman2006bifurcation}
A.~Friedman and B.~Hu.
\newblock Bifurcation from stability to instability for a free boundary problem arising in a tumor model.
\newblock {\em Archive for Rational Mechanics and Analysis}, 180:293--330, 2006.

\bibitem{friedman2008stability}
A.~Friedman and B.~Hu.
\newblock Stability and instability of liapunov-schmidt and hopf bifurcation for a free boundary problem arising in a tumor model.
\newblock {\em Transactions of the American Mathematical Society}, 360(10):5291--5342, 2008.

\bibitem{angio2}
A.~Friedman and K.~Y. Lam.
\newblock Analysis of a free-boundary tumor model with angiogenesis.
\newblock {\em Journal of Differential Equations}, 259:7636--7661, 2015.

\bibitem{friedman1999analysis}
A.~Friedman and F.~Reitich.
\newblock Analysis of a mathematical model for the growth of tumors.
\newblock {\em Journal of Mathematical Biology}, 38:262--284, 1999.

\bibitem{friedman2001symmetry}
A.~Friedman and F.~Reitich.
\newblock Symmetry-breaking bifurcation of analytic solutions to free boundary problems: an application to a model of tumor growth.
\newblock {\em Transactions of the American Mathematical Society}, 353(4):1587--1634, 2001.

\bibitem{friedman2001nonlinear}
Avner Friedman and Fernando Reitich.
\newblock Nonlinear stability of a quasi-static {Stefan} problem with surface tension: a continuation approach.
\newblock {\em Annali della Scuola Normale Superiore di Pisa-Classe di Scienze}, 30(2):341--403, 2001.

\bibitem{hao2012bifurcation}
W.~Hao, J.~D. Hauenstein, B.~Hu, Y.~Liu, A.~J. Sommese, and Y.~T. Zhang.
\newblock Bifurcation for a free boundary problem modeling the growth of a tumor with a necrotic core.
\newblock {\em Nonlinear Analysis: Real World Applications}, 13(2):694--709, 2012.

\bibitem{he2021existence}
W.~He and R.~Xing.
\newblock The existence and linear stability of periodic solution for a free boundary problem modeling tumor growth with a periodic supply of external nutrients.
\newblock {\em Nonlinear Analysis: Real World Applications}, 60:103290, 2021.

\bibitem{he2022linear}
W.~He, R.~Xing, and B.~Hu.
\newblock Linear stability analysis for a free boundary problem modeling multilayer tumor growth with time delay.
\newblock {\em SIAM Journal on Mathematical Analysis}, 54(4):4238--4276, 2022.

\bibitem{he2022thelinear}
W.~He, R.~Xing, and B.~Hu.
\newblock The linear stability for a free boundary problem modeling multilayer tumor growth with time delay.
\newblock {\em Mathematical Methods in the Applied Sciences}, 45(11):7096--7118, 2022.

\bibitem{huang2024periodic}
Y.~Huang and B.~Hu.
\newblock Periodic solution for a free-boundary tumor model with small diffusion-to-growth ratio.
\newblock {\em Journal of Differential Equations}, 399:252--280, 2024.

\bibitem{angio1}
Y.~Huang, Z.~Zhang, and B.~Hu.
\newblock Bifurcation for a free-boundary tumor model with angiogenesis.
\newblock {\em Nonlinear Analysis: Real World Applications}, 35:483--502, 2017.

\bibitem{huang2021asymptotic}
Y.~Huang, Z.~Zhang, and B.~Hu.
\newblock Asymptotic stability for a free boundary tumor model with angiogenesis.
\newblock {\em Journal of Differential Equations}, 270:961--993, 2021.

\bibitem{liu2025periodic}
J.~Liu and B.~Hu.
\newblock Periodic solution for a free-boundary vascular tumor model.
\newblock {\em Discrete and Continuous Dynamical Systems-B}, 30(12):5073--5093, 2025.

\bibitem{lowengrub2}
J.~S. Lowengrub, H.~B. Frieboes, F.~Jim, Y.~L. Chuang, X.~Li, P.~Macklin, S.~M. Wise, and V.~Cristini.
\newblock Nonlinear modelling of cancer: bridging the gap between cells and tumours.
\newblock {\em Nonlinearity}, 23:1--91, 2010.

\bibitem{spill2015mesoscopic}
F.~Spill, P.~Guerrero, T.~Alarcon, P.~K. Maini, and H.~M. Byrne.
\newblock Mesoscopic and continuum modelling of angiogenesis.
\newblock {\em Journal of Mathematical Biology}, 70(3):485--532, 2015.

\bibitem{wu2025optimal}
Y.~Wu, X.~E. Zhao, R.~Leander, and W.~Ding.
\newblock Optimal control for a free boundary tumor growth model.
\newblock {\em Evolution Equations and Control Theory}, 14(6):1534--1564, 2025.

\bibitem{zhao2025analysis}
X.~E. Zhao.
\newblock Analysis and optimization of tumor inhibitor treatments in a free boundary tumor growth model.
\newblock {\em Nonlinear Analysis: Real World Applications}, 86:104406, 2025.

\bibitem{zhao2020impact}
X.~E. Zhao and B.~Hu.
\newblock The impact of time delay in a tumor model.
\newblock {\em Nonlinear Analysis: Real World Applications}, 51:103015, 2020.

\bibitem{zhao2020symmetry}
X.~E. Zhao and B.~Hu.
\newblock Symmetry-breaking bifurcation for a free-boundary tumor model with time delay.
\newblock {\em Journal of Differential Equations}, 269(3):1829--1862, 2020.

\bibitem{zhao2025determination}
X.~E. Zhao and J.~Shi.
\newblock On determination of the bifurcation type for a free boundary problem modeling tumor growth.
\newblock {\em Journal of Differential Equations}, 436:113352, 2025.

\bibitem{zhao2024optimal}
X.~E. Zhao, Y.~Wu, R.~Leander, W.~Ding, and S.~Lenhart.
\newblock Optimal control of treatment in a free boundary problem modeling multilayered tumor growth.
\newblock {\em arXiv preprint arXiv:2410.14114}, 2024.

\end{thebibliography}
\end{document}